\documentclass[]{amsart}
\usepackage{amsmath,amssymb,mathrsfs,graphicx,pict2e}
\usepackage{amsthm}
\usepackage{type1cm}
\usepackage{color}
\usepackage[top=30truemm,bottom=30truemm,left=25truemm,right=25truemm]{geometry} 
\begin{document}
\newcommand{\val}{\mathrm{val}}
\newcommand{\VAL}{\mathrm{Val}}
\newcommand{\rk}{\mathrm{rk}}
\newcommand{\rkexp}{\mathrm{rk}_{\mathrm{exp}}}
\newcommand{\N}{\mathbb{N}}
\newcommand{\Z}{\mathbb{Z}}
\newcommand{\Q}{\mathbb{Q}}
\newcommand{\R}{\mathbb{R}}
\newcommand{\C}{\mathbb{C}}
\newcommand{\I}{\mathrm{I}}
\newcommand{\II}{\mathrm{II}}
\newcommand{\III}{\mathrm{III}}
\newcommand{\IV}{\mathrm{IV}}
\newcommand{\V}{\mathrm{V}}
\newcommand{\VI}{\mathrm{VI}}
\newcommand{\VII}{\mathrm{VII}}
\newcommand{\VIII}{\mathrm{VIII}}
\newcommand{\IX}{\mathrm{IX}}

\newcommand{\E}{\mathrm{E}}

\theoremstyle{definition}

\newtheorem{thm}{Theorem}[section]
\newtheorem*{mainthm}{Theorem~\ref{thm1}}
\newtheorem{cor}[thm]{Corollary}
\newtheorem{lem}[thm]{Lemma}
\newtheorem{prop}[thm]{Proposition}
\newtheorem{defi}[thm]{Definition}
\theoremstyle{remark}
\newtheorem{rem}[thm]{\bf Remark}
\newtheorem{exa}[thm]{\bf Example}

\title{Tropicalization of 1-tacnodal curves on toric surfaces}
\author{Takuhiro TAKAHASHI}
\address{Mathematical Institute, Tohoku University, Aoba, Sendai, Miyagi, 980-8578, Japan}
\email{takahashi.takuhiro.q7@dc.tohoku.ac.jp}
\maketitle
\begin{abstract}

A degeneration of a singular curve on a toric surface, 
called a tropicalization, was constructed by E. Shustin. 
He classified the degeneration of 1-cuspidal curves 
using polyhedral complexes called tropical curves. 
In this paper, 
we define a tropical version of a 1-tacnodal curve, that is, 
a curve having exactly one singular point 
whose topological type is $A_3$, and by applying 
the tropicalization method, 
we classify tropical curves which correspond to 1-tacnodal curves.
\end{abstract}

\tableofcontents

\section{Introduction}

Tropical geometry is a modern study area on a polyhedral complex, 
which can be obtained as the non-linear locus of 
a polynomial over the max-plus algebra. 
Among previous studies on tropical geometry, 
the most famous result is an application to the enumerative geometry 
on toric surfaces by G. Mikhalkin \cite{M}. 
T. Nishinou and B. Siebert \cite{NS} also showed that 
the enumerative problem 
on toric varieties equals to the enumeration of 
a certain type of tropical curves. 
These results are obtained 
by connecting tropical geometry with degeneration of nodal curves.

In order to apply tropical geometry to general singular curves, 
E. Shustin \cite{S} presented a degeneration of a curve, 
called a \textit{tropicalization},  
and showed that the tropicalization of a curve 
which has only one singular point whose topological type is $A_2$ 
(he called such a curve a \textit{1-cuspidal curve} for simplicity)
is related to a certain tropical curve, 
called a tropical 1-cuspidal curve. 
Furthermore, using the theory of patchworking, 
he showed that the enumeration of 1-cuspidal curves 
reduced into that of the tropical 1-cuspidal curves.

In this paper, 
we apply the tropicalization method to 1-tacnodal curves, 
that is, 
curves which have exactly one singular point 
whose topological type is $A_3$, on a toric surface,
and classify them using tropical curves.

To state our result, we prepare some terminology in tropical geometry. 
Let $F$ be a polynomial in two variables 
over the field of convergent Puiseux series over $\C$, 
denoted by $K:=\C\{\{t\}\}$. 
Then we can define a valuation 
$\val: K^{*} := K\setminus \{ 0 \} \to \R$ 
as follows. 
For a given element $b(t) \in K^*$, 
take the minimal exponent $q$ of $b(t)$ in $t$, 
then define $\val(b(t)):=-q$. 
We set 
\[
\VAL: (K^{*})^2 \to \R^2 ; (z,w) \mapsto (\val(z),\val(w)). 
\]
We call the closure 
\[
T_F:=\mathrm{Closure}(
\VAL(\{ F=0 \} \cap (K^*)^2))
\subset \R^2
\]
of the curve defined by $F$ in $(K^*)^2$ 
the \textit{tropical amoeba} defined by $F$. 

Each tropical amoeba $T$ has a positive integer $\rk(T)$ 
called a \textit{rank}, 
which, roughly speaking, is the dimension of 
the space of tropical curves which are combinatorially same as $T$. 
The formal definition of the rank will be given in Subsection 2.1. 
We will also give the definition of a {\it tropical 1-tacnodal curve} in  
Definition~\ref{tropA3} 
as a tropical analogy of a classical 1-tacnodal curve.

The following statement is the main result in this paper. 

\begin{mainthm}\it
Let $F \in K[z,w]$ be a polynomial 
which defines an irreducible 1-tacnodal curve. 
If the rank of the tropical amoeba $T_F$ defined by $F$ 
is more than or equal to 
the number of the lattice points of the Newton polytope of 
$F$ minus four, 
then $T_F$ is a tropical 1-tacnodal curve. 
\end{mainthm}

In \cite{S}, 
Shustin proved the 1-cuspidal version of this theorem 
and the statement that 
``\textit{for each tropical 1-cuspidal curve, 
we can calculate the number of classical curves 
degenerated into the 1-cuspidal curve 
by using the patchworking method}'', which means that 
the enumeration of 1-cuspidal curves on toric surfaces 
can be carried out by using tropical 1-cuspidal curves. 
The original aim of this paper is to enumerate 1-tacnodal curves 
on toric surfaces by the same method. 
However, 
it does not work unlike the studies for nodal and 1-cuspidal curves 
because the criterion of patchworking developed 
by Shustin \cite{Sglu, S} 
cannot be used in this case. 
We will discuss this in Remark~\ref{further} below.

We organize this paper as follows. 
In Subsection 2.1, 
we define some basic terminology on tropical geometry 
such as the dual subdivision and the rank of a tropical curve, 
and introduce a lemma on the rank of the tropical curves 
proved by E. Shustin.
In Subsection 2.2, 
we consider a necessary and sufficient condition 
for a complex curve to have a tacnode, 
and estimate the dimension of the space of the 1-tacnodal curves 
on a toric surface. 
In Subsection 2.3, 
we summarize the tropicalization and its refinement. 
In Section 3, we define tropical 1-tacnodal curves 
and discuss polytopes appearing 
in their dual subdivisions. 
We also study a reduced curve 
associated with a tropical 1-tacnodal curve. 
In Subsections 4.1 and 4.2, before the proof of Theorem~\ref{thm1}, 
we prepare some definition and lemmata 
on relation between singular curves and 
their Newton polytopes. 
The proof of Theorem~\ref{thm1} 
is carried out from Subsections 4.3 to 4.5.

\subsection*{Acknowledgment}
The author appreciates 
Professor Masaharu Ishikawa and Professor Takeo Nishinou 
for valuable discussions. 
He is also grateful to Professor Eugenii Shustin 
for giving comments on refinements of tropicalization. 

\section{Preliminaries}

A set in $\R^2$ is a \textit{(lattice) polyhedron} 
if it is the intersection of a finite number of half-spaces in 
$\mathbb{R}^2$ 
whose vertices are contained in the lattice 
$\mathbb{Z}^2 \subset \mathbb{R}^2$.  
A set is a \textit{polytope} in $\mathbb{R}^2$ if 
it is a compact polyhedron, 
that is, 
the convex hull of a finite number of lattice points. 
We call a facet of a polytope an \textit{edge}. 
Similarly, we call the 0-dimensional sub-polytope 
obtained as the corner of a polytope 
a \textit{vertex}. 
The \textit{boundary} 
$\partial \Delta$ 
is the union of all facets of $\Delta$. 
The \textit{interior} 
$\mathrm{Int}\Delta$ 
is defined by $\Delta \setminus \partial \Delta$. 

A polytope is said to be \textit{parallel} if the opposite edges 
have the same directional vector (up to orientation) 
and the same lattice length. 
A polytope is called an \textit{$m$-gon} if the number of its edges is $m$. 

Let 
$\Delta \subset \mathbb{R}^2$ 
be a polytope. 
We denote the interior lattice points of 
$\Delta$, $\mathrm{Int}\Delta$ and $\partial \Delta$ as  
$\Delta_{\Z}$, $\mathrm{Int}\Delta_{\Z}$ and $\partial \Delta_{\Z}$, 
respectively. 
That is,
\[
\Delta_{\Z}:=\Delta \cap \Z^2, \ \ \ 
\mathrm{Int}\Delta_{\Z}
:=\mathrm{Int}\Delta \cap \Z^2, \;\;\;
\partial \Delta_{\Z}
:=\partial \Delta \cap \Z^2. 
\]

For a polytope $\Delta \subset \R^2$, 
we can construct 
a polarized toric surface associated with 
$\Delta$ over $\C$, 
denoted by 
$(X(\Delta),D(\Delta))$, 
where $D(\Delta)$ is the polarization on $X(\Delta)$ 
associated with $\Delta$.

\subsection{Basics of tropical plane curves}

Throughout this paper, 
$K:=\mathbb{C}\{\{ t \}\}$ 
represents the field of convergent Puiseux series over $\mathbb{C}$. 
The field $K$ admits a non-Archimedean valuation 
\[
\val : K^{*} \to \Q; 
\sum_{k=k_0}^{\infty}b_k t^{\frac{k}{N}} \mapsto
-\frac{k_0}{N},   
\]
where $b_{k_0} \neq 0$. 
For a polynomial 
\[
F(z,w)=\sum_{(i,j) \in \Delta_{\mathbb{Z}}}c_{ij} z^i w^j 
\in K[z,w],  
\]
the sets 
\[
\mathrm{Supp}(F)
:=\{ (i,j) \in \mathbb{R}^2 ; c_{ij} \neq 0 \} 
\quad \text{and} \quad
N_F
:=\mathrm{Conv}(\mathrm{Supp}(F)) \subset \mathbb{R}^2
\]
are called 
the \textit{support} of $F$ 
and the \textit{Newton polytope} of $F$, respectively, 
where $\mathrm{Conv}(A)$ is the convex hull of $A$ in $\R^2$. 
In this paper, we always assume that 
the dimension of any polytope is $2$. 
Then the map $\VAL$ is defined by 
\[
	\VAL : (K^*)^2 \to \mathbb{R}^2 ;\ 
 	(z,w) \mapsto (\val (z), \val (w)). 
\]
The set $T_F$ is defined by 
\[
 	T_F
 	:=\mathrm{Closure}\bigl({\VAL(\{p \in (K^*)^2 ; F(p)=0\})}\bigr)
 	\subset \mathbb{R}^2,  
\]
where  
$\mathrm{Closure}(A)$ is the closure of $A$ 
with usual topology on $\mathbb{R}^2$. 
The set $T_F$ is called the \textit{tropical amoeba} 
defined by $F$. 

On the other hand, 
for $F$, 
the \textit{tropical polynomial} $\tau_F$ is defined by   
\[
\tau_F(x,y)
:=\max\{ \val(c_{ij})+ix+jy; (i,j) \in \mathrm{Supp}(F) \} 
\]
over the max-plus algebra. 
The non-linear locus of a polynomial over the max-plus algebra 
in two variables  
is called a \textit{tropical plane curve}, 
which is a 1-dimensional polyhedral complex in $\R^2$. 
It is known as Kapranov's Theorem \cite{K} 
that the tropical amoeba $T_F$ coincides with 
the tropical plane curve of the tropical polynomial $\tau_F$. 
Hence, $T_F$ has 
the structure of a 1-dimensional polyhedral complex in $\R^2$. 
The 1-simplex and 0-simplex of $T_F$ is called 
an \textit{edge} and a \textit{vertex} of $T_F$, respectively.

An edge $E$ of $T_F$ corresponds to the intersection of 
two linearity domains of $\tau_F$. 
If one of the linearity domain 
is defined by a term $a_{ij}+ix+jy$ of $\tau_F$ 
and the other is defined by a term $a_{i'j'}+i'x+j'y$ of $\tau_F$, 
then the \textit{weight} $w(E)$ of the edge $E$ is defined as 
the greatest common divisor of $i-i'$ and $j-j'$.

Now we introduce a subdivision of the Newton polytope $N_F$ 
which is dual to the tropical amoeba $T_F$. 
Let $\nu_F:N_F \to \R$ be the discrete Legendre transform of $\tau_F$, 
which is a continuous concave PL-function 
(see, for example, \cite[Chapter 1.5]{G}). 
Then we obtain a subdivision of $N_F$ consisting of 
the following three kinds of polytopes from $\nu_F$: 
\begin{itemize}
\item 
linearity domains of $\nu_F$: 
$\Delta_1, \cdots ,\Delta_N$,
\item 
1-dimensional polytopes: 
$\sigma_{ij}:=\Delta_i \cap \Delta_j \neq \emptyset, 
\neq \{ \mathrm{pt} \}$,
\item 
0-dimensional polytopes: 
$\Delta_{i_1} \cap \Delta_{i_2} \cap \Delta_{i_3} \neq \emptyset$. 
\end{itemize}
These polytopes give a subdivision of $N_F$, which we denote by $S_F$.
We call a $1$-dimensional and a $0$-dimensional sub-polytope of $N_F$ 
contained in $S_F$ 
an \textit{edge} and a \textit{vertex}, respectively.

The following claim is in \cite[Proposition 3.11]{M}. 

\begin{thm}[Duality Theorem]\label{DualThm}\it
There exists a correspondence between a tropical curve
$T_F$ 
with the weight $w(E)$ on each edge $E \subset T_F$ 
and the corresponding subdivision $S_F$ of $N_F$ 
in the following sense: 
 \begin{itemize}
 \item[(1)]
 the components of 
 $\mathbb{R}^2 \setminus T_F$ 
 are in 1-to-1 correspondence 
 with the vertices of the subdivision 
 $S_F$, 
 \item[(2)]
 the edges of 
 $T_F$ 
 are in 1-to-1 correspondence 
 with the edges of $S_F$ 
 so that an edge 
 $E \subset T_F$ 
 is dual to an orthogonal edge of 
 $S_F$ 
 having the lattice length equal to 
 $w(E)$ 
 \item[(3)]
 the vertices of 
 $T_F$ 
 are in 1-to-1 correspondence with 
 the polytopes 
 $\Delta_{1}, \dots, \Delta_{N}$
 of $S_F$ so that the valency of a vertex of 
 $T_F$ 
 is equal to the number of sides of the corresponding polytope.
 \end{itemize}
\end{thm}

We call the set $S_F$ 
the \textit{dual subdivision} of $T_F$. 
By Theorem \ref{DualThm}, 
we can study any plane tropical curve 
by using the dual subdivision of the corresponding Newton polytope. 

\vspace{3mm}

Next, we discuss the dimension of 
the space of tropical curves. 
For a given polytope $\Delta$, 
let 
$\mathfrak{T}(\Delta)$ 
denote the set of tropical curves 
which are defined by polynomials in two variables 
over the max-plus algebra with Newton polytope $\Delta$. 
Let $S$ be the dual subdivision of 
$T \in \mathfrak{T}(\Delta)$ and 
define the 
\textit{rank} of the tropical curve $T$ (or of $S$) 
as 
\[
\rk(T)
:=\rk(S)
:=\dim \{ T' \in \mathfrak{T}(\Delta) ; S=S' \}, 
\]
where $S'$ is the dual subdivision of $T'$. 
By \cite[Lemma 3.14]{M}, 
the set 
$\{ T' \in \mathfrak{T}(\Delta) ; S=S' \}$
is a polyhedron in 
$\mathbb{R}^M$ 
for some positive integer $M$. 
Thus, the definition of the rank is well-defined. 

Let $\Delta_1 ,\dots, \Delta_N$ be the $2$-dimensional polytopes of $S$. 
According to \cite{S}, 
we define the 
\textit{expected rank} of the tropical curve $T$ (or of $S$) 
as 
\[
 	\rkexp(T)
 	:=\rkexp(S)
 	:=\sharp V(S) -1 - \sum_{k=1}^N (\sharp V(\Delta_{k}) -3), 
\]
where 
$V(S)$ is the set of vertices of $S$ and 
$V(\Delta_k)$ is the set of vertices of $\Delta_k$. 

\begin{defi} 
A lattice subdivision of a polytope is a 
\textit{TP-subdivision} 
if the subdivision consists of only 
triangles and parallelograms. 
\end{defi}

We remark that, this definition is same as 
the definition of the 
\textit{nodal subdivision} in \cite[Subsection 3.1]{S} 
except the condition on the boundary 
$\partial \Delta$. 

For any subdivision $S$, 
we denote 
the number of $\ell$-gons and
the number of parallel $(2m)$-gons contained in $S$ 
as $N_{\ell}$ and $N'_{2m}$, respectively.  

The following statement is in \cite[Lemma 2.2]{S}. 

\begin{lem}[Shustin \cite{S}]\label{rklem}\it
The difference 
\[ d(T) := \rk(T) - \rkexp(T) \]
of a tropical curve $T$ satisfies $d(T) \ge 0$. 
Moreover, for the dual subdivision $S$ of $T$, 
the difference $d(T)$ satisfies 
\begin{itemize}
\item 
$d(T)=0$ if $S$ is a TP-subdivision and 
\item 
$0 \le 2 d(T) \le \mathcal{N}_S$ 
otherwise,  
\end{itemize}
where 
\begin{align*}
\mathcal{N}_S
 &:=\sum_{m \ge 2}((2m-3)N_{2m}-N'_{2m})+\sum_{m \ge 2}((2m-2)N_{2m+1})-1 \\
 &=\sum_{\ell \ge 3}(\ell -3) N_{\ell} - \sum_{m \ge 2}N'_{2m} -1.
\end{align*}
\end{lem}

\subsection{Some remarks on 1-tacnodal curves}

In this paper, 
a curve on a projective surface is called a \textit{1-tacnodal curve} 
if the curve has exactly one singular point 
at a smooth point of the surface 
whose topological type is $A_3$. 
The term ``tacnode" means $A_3$-singularity.
In this subsection we prepare some lemmata 
related to 1-tacnodal curves.

For a polynomial $f$ and $p \in \C^2$, 
we use the notations 
$f_x(p)=\frac{\partial f}{\partial x}(p), 
f_y(p)=\frac{\partial f}{\partial y}(p)$ 
and so on. 
We set 
$\mathrm{Hess}(f)(p)=f_{xx}(p)f_{yy}(p)-f_{xy}(p)^2$ and 
\begin{align*}
K(f)(p):=-f_{xy}(p)^3f_{xxx}(p)+
&3f_{xx}(p) f_{xy}(p)^2 f_{xxy}(p)\\
&-3f_{xx}(p)^2 f_{xy}(p) f_{xyy}(p)
+f_{xx}(p)^3f_{yyy}(p). 
\end{align*}

\begin{lem}\label{tacnode}\it
Suppose that a polynomial 
$f \in \C[x,y]$ 
satisfies 
$f_{xx}(p) \neq 0$.  
Then the curve 
$\{f=0\} \subset \C^2$
has a tacnode at $p$ if and only if 
$f$ satisfies 
\begin{itemize}
\item[(1)] $\displaystyle f(p)=f_x(p)=f_y(p)=0$, 
\item[(2)] $\displaystyle \mathrm{Hess}(f)(p)=0$, 
\item[(3)] $\displaystyle K(f)(p)=0$, 
\item[(4)] $\displaystyle a_{12}(p)^2-4f_{xx}(p)a_{04}(p) \neq 0$, 
\end{itemize}
where 
\begin{align*}
a_{12}(p):=&
f_{xy}(p)^2f_{xxx}(p)
-2f_{xx}(p)f_{xy}(p)f_{xxy}(p)
+f_{xx}(p)^2f_{xyy}(p), 
\\
a_{04}(p):=&
f_{xy}(p)^4f_{xxxx}(p)
-4f_{xx}(p)f_{xy}(p)^3f_{xxxy}(p)\\
&+6f_{xx}(p)^2f_{xy}(p)^2f_{xxyy}(p)
-4f_{xx}(p)^3f_{xy}(p)f_{xyyy}(p)
+f_{xx}(p)^4f_{yyyy}(p).
\end{align*}

\end{lem}
\begin{proof}
For simplicity, we assume that $p$ is the origin $(0,0)$ of $\C^2$.
First, if the origin is a singular point then we can represent $f$ as 
\[
f=Ax^2+ Bxy +Cy^2 +(\text{higher terms}), 
\]
where $(A,B,C)=(f_{xx}(0,0)/2,f_{xy}(0,0),f_{yy}(0,0)/2)$. 
If $\mathrm{Hess}(f)(0,0) \neq 0$, 
then the origin is an $A_1$-singularity of $\{ f=0 \}$. 
Therefore $\mathrm{Hess}(f)(0,0)=0$ 
for the origin to be an $A_3$-singularity. 
Then we can rewrite $f$ as
\[
f=\frac{1}{4A}(2Ax+By)^2 + (\text{higher terms}). 
\]
The tangent line of $\{ f=0 \}$ at the origin is defined by 
\[ f_{xx}(0,0)x+f_{xy}(0,0)y=0. \]
Now we define new coordinates $(u,v)$ as 
\[
\begin{pmatrix}
u\\
v\\
\end{pmatrix}
=
\begin{pmatrix}
 f_{xx}(0,0) & f_{xy}(0,0)\\
 0           & 1          \\
\end{pmatrix}
\begin{pmatrix}
x\\
y\\
\end{pmatrix} 
\]
and set 
\[ 
\hat{f}(u,v):=f(x(u,v),y(u,v)). 
\]
Note that 
the condition 
$f(0,0)=f_x(0,0)=f_y(0,0)=\mathrm{Hess}(f)(0,0)=0$ 
is equivalent to 
$\hat{f}(0,0)=\hat{f}_u(0,0)=\hat{f}_v(0,0)=\mathrm{Hess}(\hat{f})(0,0)=0$.

By direct computation, we obtain the equalities: 
\begin{equation}\label{newco}\tag{*}
\begin{split}
&\hat{f}_{uu}(0,0)=\frac{1}{f_{xx}(0,0)}, \\
&\hat{f}_{uv}(0,0)=0, \\
&\hat{f}_{vv}(0,0)=\frac{1}{f_{xx}(0,0)}\mathrm{Hess}(f)(0,0), \\
&\hat{f}_{uvv}(0,0)=\frac{1}{f_{xx}(0,0)^3}a_{12}(0,0), \\
&\hat{f}_{vvv}(0,0)= \frac{1}{f_{xx}(0,0)^3}K(f)(0,0), \\
&\hat{f}_{vvvv}(0,0)=\frac{1}{f_{xx}(0,0)^4}a_{04}(0,0).
\end{split}
\end{equation}

From the properties of the Newton diagram of a plane curve singularity 
\cite{Kou}, 
the condition that 
the singularity at the origin is $A_3$ can be rewritten as 
\[ 
\hat{f}_{uv} (0,0)=
\hat{f}_{vv} (0,0)=
\hat{f}_{vvv}(0,0)=0, \;\;
\hat{f}_{uu}(0,0) \neq 0
\]
and 
\[
\hat{f}_{uvv}(0,0)^2-4\hat{f}_{uu}(0,0)\hat{f}_{vvvv}(0,0) \neq 0
\]
on the new coordinate system. 
By~\eqref{newco}, 
these conditions coincide with the conditions in the assertion.
\end{proof}

For $\mu=1,3$, 
let $U(\Delta,A_{\mu})$ denote a locally closed subvariety 
in the complete linear system $|D(\Delta)|$ of $D(\Delta)$ 
which parametrizes the set of curves having exactly one singular point 
whose topological type is $A_{\mu}$. 
Let $V(\Delta,A_{\mu})$ be the closure of 
$U(\Delta,A_{\mu})$ in $|D({\Delta})|$.

\begin{cor}\it \label{dimension}
If $V(\Delta,A_3)$ is non-empty 
then $\dim V(\Delta,A_3) \ge \sharp \Delta_{\Z}-4$. 
\end{cor}
\begin{proof}
For $\mu =1, 3$, 
we set 
\[
\Sigma(\Delta,A_{\mu})
:=\{(C,p) ; \text{$p$ is a singular point of $C$} \}
\subset U(\Delta, A_{\mu}) \times X(\Delta)
\subset |D(\Delta)| \times X(\Delta). 
\]
For a curve $C \in V(\Delta, A_{\mu})$, 
we choose a local coordinate system $(x,y)$ of $X(\Delta)$ 
around the singular point $p=(x_0,y_0) \in C$. 
Let $f$ be a defining polynomial of $C$. 
By Lemma~\ref{tacnode}, 
$\Sigma(\Delta, A_3)$ is locally defined by 
\[\label{tacnodeeq}\tag{**}
f(x_0,y_0) 
= f_x(x_0,y_0) 
= f_y(x_0,y_0) 
= \mathrm{Hess}(f)(x_0,y_0) 
= K(f)(x_0,y_0) =0. 
\]
Note that,
by \cite[Theorem (1.49)]{HM}, 
the dimension of the Severi variety $V(\Delta, A_1)$ satisfies 
\[
\dim V(\Delta, A_1) 
= \dim \Sigma(\Delta, A_1) 
= \sharp \Delta_{\Z}-1-1
\]
and $\Sigma(\Delta, A_1)$ is defined by the first 
three equations of (\ref{tacnodeeq}).
Therefore, we obtain 
\[
\dim V(\Delta, A_3) 
\ge \dim \Sigma (\Delta, A_3)
\ge \sharp \Delta_{\Z} -1-3 
= \sharp \Delta_{\Z}-4. 
\]
\end{proof}

\subsection{Tropicalization of curves}\label{troplicalization}

We briefly introduce the \textit{tropicalization} of a curve 
and its refinement 
(see \cite[Section 3]{S} for more details). 

Let $F \in K[z,w]$ be a reduced polynomial 
which defines a curve $C \subset X(N_F)$.  
Set $\Delta=N_F$ and let $T_F$ be the tropical amoeba defined by $F$ 
introduced in Section~2.1 and $S_F$ be the dual subdivision of $T_F$. 
We consider the 3-dimensional unbounded polyhedron 
\[
\check{\Delta}_F
:=\mathrm{Conv}
\{ (i,j,t)\in \mathbb{R}^2 \times \mathbb{R}
;t \ge \nu_F(i,j)  \}
\subset \mathbb{R}^3. 
\]
We remark that a compact facet 
$\check{\Delta}_i$ 
of $\check{\Delta}_F$ 
corresponds to a 2-dimensional polytope 
$\Delta_i$ 
in $S_F$ 
by the projection 
$\check{\Delta}_F \subset \R^2 \times \R \to \R^2$. 

We then obtain 
a toric flat morphism $X(\check{\Delta}_F) = \mathfrak{X} \to \C$ 
from the toric 3-fold associated with $\check{\Delta}_F$ to the 
complex line, which is called a \textit{toric degeneration}.  
A generic fiber $\mathfrak{X}_t$ 
is isomorphic to $X(\Delta)$, 
and its central fiber $\mathfrak{X}_0$ 
is isomorphic to 
$\bigcup_{i=1,\dots,N}X(\Delta_i)$
(see \cite[Section 3]{NS} for more details). 
Let $D \subset \C$ be a small disk centered at the origin.
We regard the indeterminate $t$ of $K$ as the variable in 
$D^*:=D \setminus \{ 0 \}$. 
Then we can get an analytic function $F(t;z,w)$ in three variables.
From this analytic function, 
we obtain an equisingular family on the toric surface $X(\Delta)$ 
\[
\{ C^{(t)}:= \mathrm{Closure}(\{ F(t;z,w)=0 \}) \}_{t \in D^*}. 
\]
The limit $ C^{(0)} $ of this family is constructed as follows:
For each $i=1,\dots ,N$, 
a complex polynomial $f_i \in \C[z,w]$ 
whose Newton polytope is $\Delta_i \in S_F$ 
is induced from the face function of $F$ on $\check{\Delta}_i$ 
by the transformation induced by the projection from $\check{\Delta}_i$ 
to $\Delta_i$.
The union of these curves is the limit $C^{(0)}$, 
which is a curve on the central fiber $\mathfrak{X}_0$ 
of the toric degeneration. 
The limit $ C^{(0)} $ is called a \textit{tropicalization} of $C$. 

For each singular point $z$ of $C$, 
there exists a continuous family of singular points 
$\{z_t\}$ for $t \in D^*$, 
where $z_t \in C^{(t)}$, 
and this family defines a section 
$s : D^* \to X(\check{\Delta}_F)$.
If the limit $s(0)=\lim_{t \to 0}s(t)$ 
does not belong to the intersection lines 
$\bigcup_{i \neq j}X(\Delta_i \cap \Delta_j)$ 
and bears just one singular point of $C^{(t)}$, 
the point $s(0)$ is called a \textit{regular singular point}. 
Otherwise it is called an \textit{irregular singular point}. 
Note that if $s(0)$ is a regular singular point then 
it is topologically equivalent to the original singularity.

If the singular point $s(0)$ is irregular, 
additional information 
can be obtained by the \textit{refinement} of the tropicalization, 
see Figure~1.
In the rest of this section, 
we explain this method briefly. 
See \cite[Subsection 3.5]{S}
for the details of the refinement.

Hereafter, we assume that $F$ defines a 1-tacnodal curve in $X(\Delta)$.
Let $\Delta_1 \in S_F$ and $\Delta_2 \in S_F$ 
be polytopes which have a common edge 
$\sigma$ of length $m \ge 2$ 
and we observe the case where 
an irregular singularity degenerates into 
the subvariety 
$X(\sigma)$ of $\mathfrak{X}_0$. 
For each $i=1,2$, 
let $f_i$ be a polynomial whose Newton polytope is $\Delta_i$ 
such that the union of curves 
$C_1 \cap C_2 \subset C^{(0)}$ 
defined by $f_1=f_2=0$ intersects 
$X(\sigma)$ at $z \in \mathfrak{X}_0$. 
In this paper, by later discussion, we can assume that, 
for each $i=1,2$, 
the polynomial $f_i$ has 
an isolated singularity at $z \in X(\sigma)$ 
and their Newton boundary intersects 
the $x$- and $y$-axes at $(m_i,0)$ and $(0,m)$, respectively, 
where the $y$-axis corresponds to $X(\sigma)$.

Find an automorphism 
$M_{\sigma} \in \mathrm{Aff}(\Z^2)$
such that 
$M_{\sigma}(\Delta)$ is contained in the right half-plane of $\R^2$ 
and 
$M_{\sigma}(\sigma)=:\sigma'$ is a horizontal segment, 
see Figure~1. 
The automorphism $M_{\sigma}$ 
induces a transformation $(x,y) \mapsto (x',y')$, by which 
we obtain a new polynomial $F'(x',y')$ from $F$. 
We can assume that $F' \in K[x',y']$ 
by multiplying a monomial.
We remark that the point $z$ corresponds to 
a root $\xi \neq 0$ of the truncation polynomial 
$F'^{\sigma'}(x',y')$ of $F'$ on $\sigma'$. 
Here 
the truncation polynomial $F^\sigma$ of a polynomial $F$ 
on a facet $\sigma$ of $N_F$ is the sum of the terms of $F$ 
corresponding to the lattice points on $\sigma$. 

Then we choose an element 
$\tau \in K$ such that 
the coefficient of $\tilde{x}^{m-1}$ in 
$\tilde{F}(\tilde{x},\tilde{y})=F'(\tilde{x} + \xi + \tau , \tilde{y})$ 
is zero. 
Moreover, the dual subdivision of the tropical amoeba 
defined by $\tilde{F}$ contains a subdivision of the triangle 
$\Delta_z:=\mathrm{Conv}\{ (m,0), (0,m_1), (0,-m_2) \}$. 
In this paper, we call the polytope $\Delta_{z}$ 
the \textit{exceptional polytope} for the irregular singularity 
$z \in \mathfrak{X}_0$. 
We remark that, the exceptional polytope is 
the union of the complements of 
the Newton diagrams of the polynomials 
$f_1$ and $f_2$ at $z \in X(\sigma)$ 
in the first quadrant of $\R^2$. 
Making the exceptional polytope $\Delta_z$ by the translation 
is an operation similar to a blowing-up of 
the $3$-fold $\mathfrak{X}$. 
We can restore the topological type of 
the irregular singularity $z$ in $X(\Delta_z)$ by this operation.

\begin{figure}[htbp]\label{refi}
\includegraphics[scale=1]{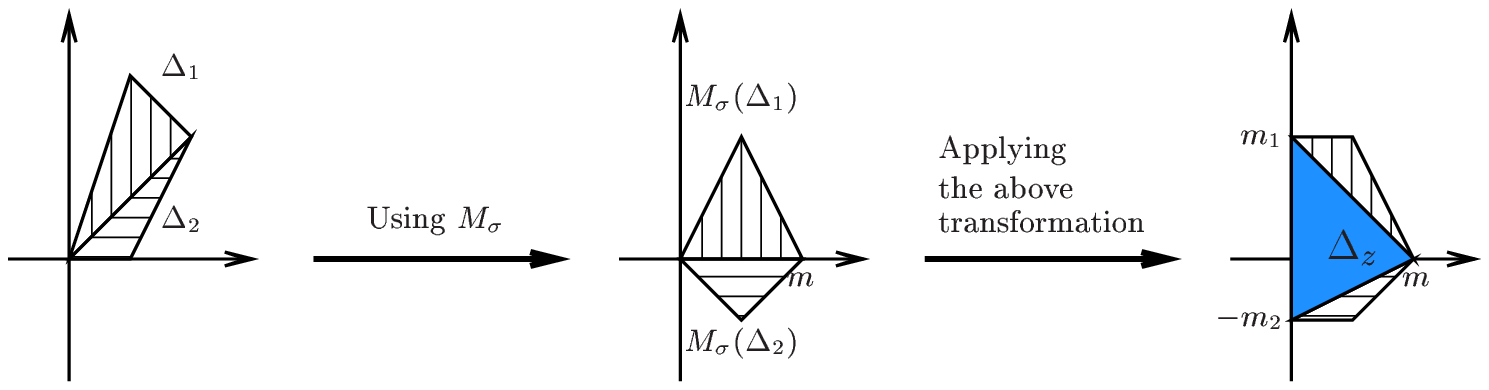}
\caption{
A refinement of a tropicalization
} 
\end{figure}


\begin{defi}\label{defpat}
For each $i=1,2$, 
let $f_i$ be a polynomial which defines $C_i$ such that 
$f_1^{\sigma} = f_2^{\sigma}$, 
and $\phi_i$ denote  
the composition of $f_i$ 
and the translation which maps $z$ to the origin of $\C^2$. 
Set 
\[ 
\Hat{\sigma}_i:=\Delta_{z} \cap N_{\phi_i}
\subset \Delta_{z}, 
\] 
where $N_{\phi_i}$ is the Newton polytope of $\phi_i$. 
We assume that $\hat{\sigma}_i$ is an edge of $\Delta_z$. 
We call a polynomial $\phi$ 
whose Newton polytope is $\Delta_{z}$
and that satisfies 
\begin{itemize}
\item[(a)] 
the coefficient of $x^{m-1}$ is zero, and 
\item[(b)] 
the truncation polynomial $\phi^{\Hat{\sigma}_i}$ 
is equal to $\phi_i$ for each edge $\hat{\sigma}_i$ of $\Delta_z$.
\end{itemize}
a \textit{deformation pattern compatible with given data} 
$(f_1,f_2,z)$. 
\end{defi}

We remark that, by the same reason as in 
\cite[Subsection 3.5]{S}, except case (E), 
if the curve defined by $F$ has only one singular point which is 
an irregular singularity and 
there does not exist a deformation pattern 
compatible with the irregular singularity 
which defines a $1$-tacnodal curve, 
then $F$ does not define a $1$-tacnodal curve. 
We will discuss what happen in case (E) in Subsection~3.4.

\section{Tropical 1-tacnodal curves}

\subsection{Definition of tropical $1$-tacnodal curves}
In this subsection, we define a tropical 1-tacnodal curve. 
We can think of it as a tropical version of a $1$-tacnodal curve, 
which is the main theorem (Theorem~\ref{thm1}) in this paper.

Set 
\begin{align*}
&\Delta_{\I}:=\mathrm{Conv}\{ (0,7),(1,0),(2,0) \},\;
\Delta_{\II}:=\mathrm{Conv}\{ (0,7),(2,0),(3,0) \},\\
&\Delta_{\III}:=\mathrm{Conv}\{ (0,0),(2,0),(1,3) \},\;
\Delta_{\IV}:=\mathrm{Conv}\{ (0,0),(2,0),(1,2) \}\\
&\Delta_{\V}:=\mathrm{Conv}\{(0,0), (4,0),(0,1) \},\;
\Delta_{\VI}:=\mathrm{Conv}\{ (1,0),(2,0),(0,3),(1,3) \},\\
&\Delta_{\VII}:=\mathrm{Conv}\{ (0,0),(1,0),(2,1),(0,1),(1,2) \},\\
&\Delta_{\VIII}:=\mathrm{Conv}\{(0,0),(1,0),(0,1),(3,3) \},\;
\Delta_{\IX}:=\mathrm{Conv}\{(0,0),(1,0),(0,1),(4,2) \}, \\
&\Delta_{\E}:=\mathrm{Conv}\{(0,0),(2,0),(0,1),(1,2)\}, 
\end{align*}
see Figure 2.

We say that a polytope $P \subset \R^2$ is 
$\mathrm{Aff}(\Z^2)$-equivalent to
(or simply, equivalent to) $P'$ 
if there exists an affine isomorphism 
$A \in \mathrm{Aff}(\Z^2)$ 
such that $A(P)=P'$, and denote it
as $P \simeq P'$.

\begin{defi}\label{tropA3}
A tropical curve $T$ is said to be 
\textit{tropical $1$-tacnodal} 
if the dual subdivision $S$ of $T$ 
contains one of the following polytopes or unions of polytopes:
\begin{itemize}
\item[(I)]
a triangle equivalent to 
$\Delta_{\I}$,

\item[(II)]
a triangle equivalent to 
$\Delta_{\II}$,

\item[(III)]
the union of a triangle equivalent to $\Delta_{\III}$ 
and a triangle with edges of lattice length $1$, $1$ and $2$ 
and without interior lattice point 
glued in such a way that 
they share the edge of lattice length $2$, 

\item[(IV)]
the union of two triangles equivalent to $\Delta_{\IV}$ 
which share the edge of lattice length $2$,

\item[(V)]
the union of two triangles equivalent to $\Delta_{\V}$ 
which share the edge of lattice length $4$,

\item[(VI)]
a parallelogram equivalent to 
$\Delta_{\VI}$,

\item[(VII)]
a pentagon equivalent to 
$\Delta_{\VII}$,

\item[(VIII)]
a quadrangle equivalent to 
$\Delta_{\VIII}$, 

\item[(IX)]
a quadrangle equivalent to 
$\Delta_{\IX}$, 

\item[(E)]
the union of a quadrangle equivalent to 
$\Delta_{\E}$ and a triangle with edges of lattice length $1$, $1$ and $2$ 
and without interior lattice point 
which share the edge of lattice length $2$, 


\end{itemize}
and the rest of $S$ consists of 
triangles of area 1/2.
\end{defi}

\begin{figure}[htbp]
\includegraphics[scale=0.6]{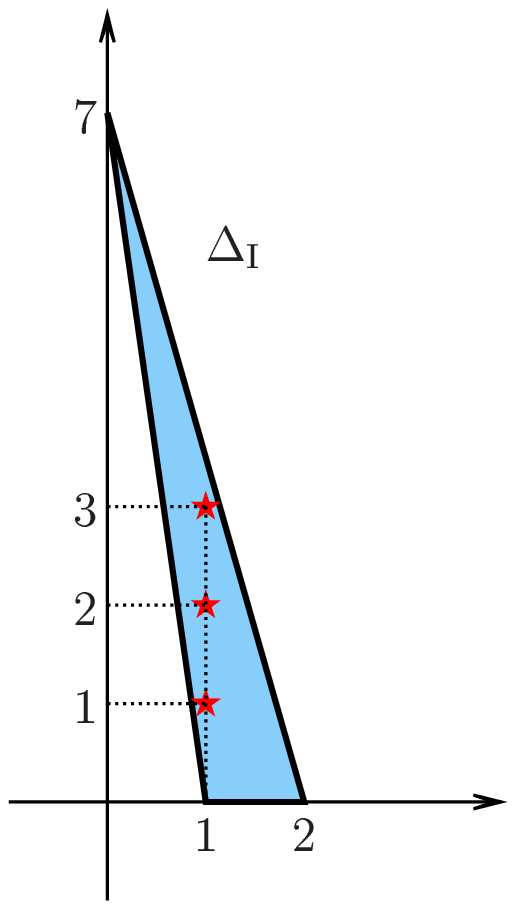}
\includegraphics[scale=0.6]{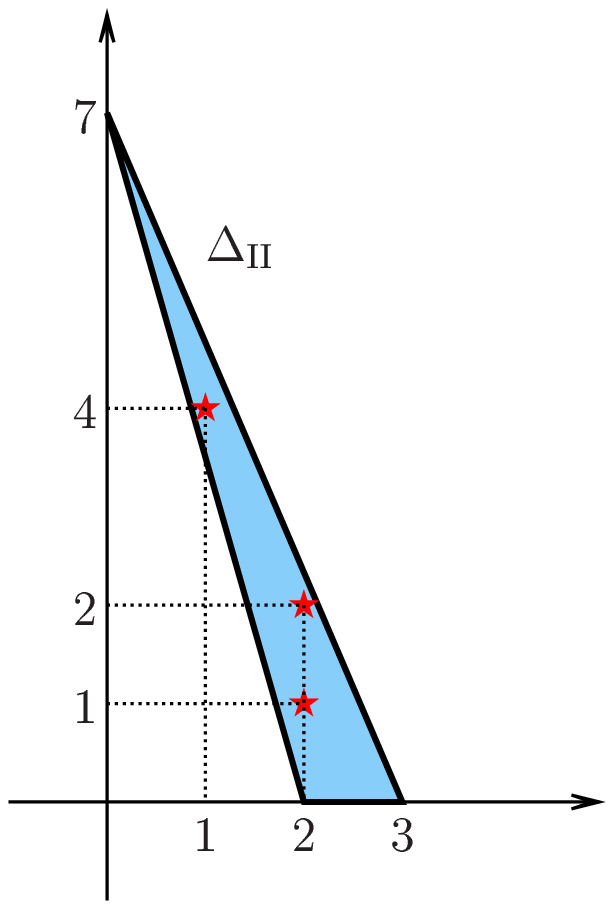}
\includegraphics[scale=0.6]{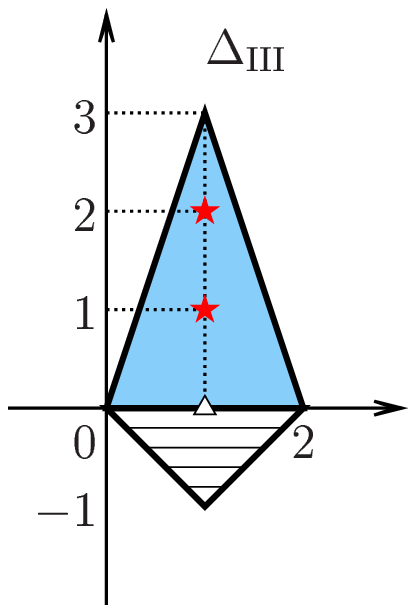}
\includegraphics[scale=0.6]{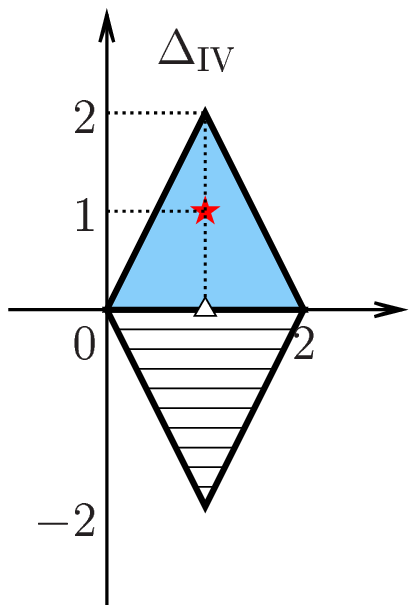}
\includegraphics[scale=0.6]{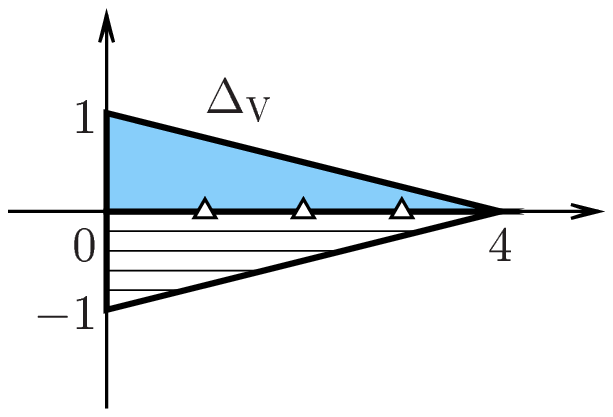}
\includegraphics[scale=0.6]{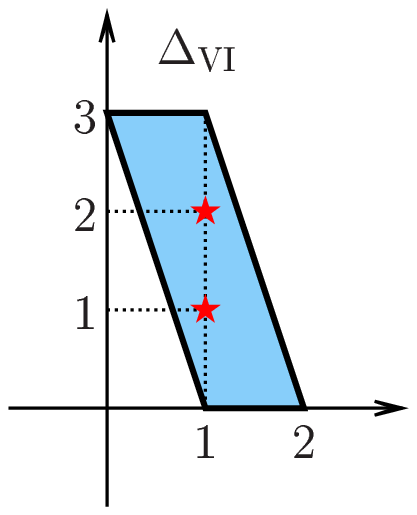}
\includegraphics[scale=0.6]{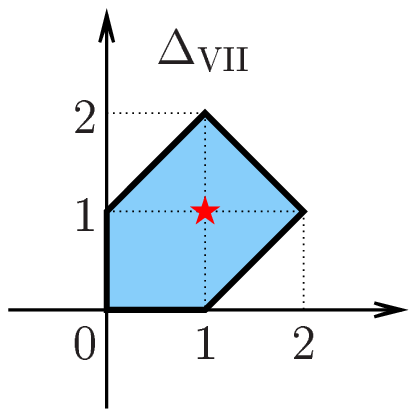}
\includegraphics[scale=0.6]{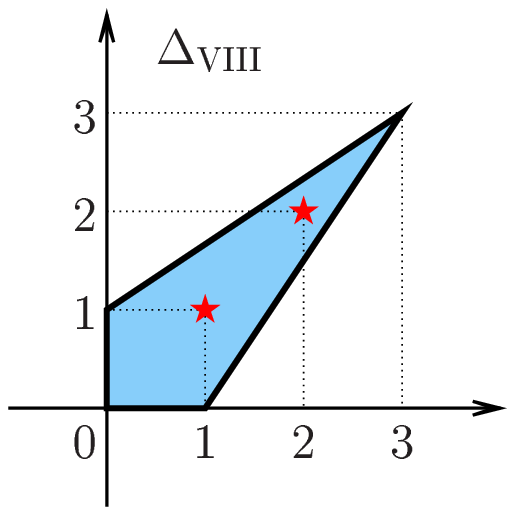}
\includegraphics[scale=0.6]{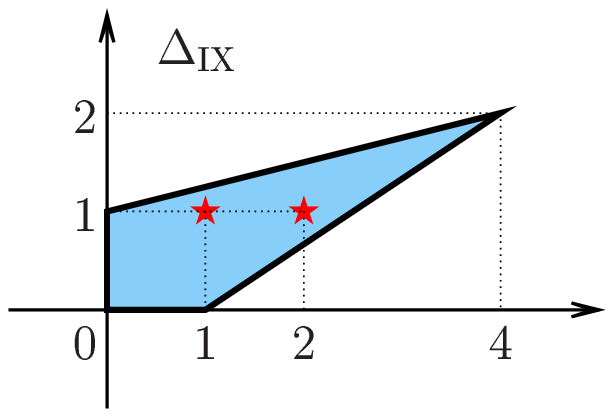}
\includegraphics[scale=0.6]{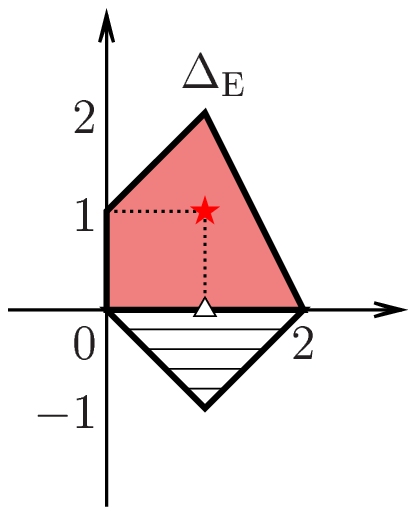}
\caption{
Polytopes in Definition \ref{tropA3}. 
The notation {\tiny $\triangle$} means a lattice point on the boundary 
which is not a vertex 
and the notation \textcolor{red}{$\star$} means an interior lattice point.
} 
\end{figure}

\subsection{Polytopes corresponding to tropical $1$-tacnodal curves}
In this subsection, 
we mention some remark on polytopes 
appearing in Definition \ref{tropA3}.

We denote an 
$m$-gon which has edges of lattice lengths 
$\ell_1, \dots, \ell_m$ 
and $I$ interior lattice points by  
\[ \Delta_{m} (I; \ell_1 ,\dots ,\ell_m ). \] 
Similarly, 
we denote a parallel $2m$-gon 
which has $m$ pairs of antipodal parallel edges of 
lattice length $\ell_1,\ldots,\ell_m$ by 
\[ \Delta^{\mathrm{par}}_{2m} (I; \ell_1 ,\dots ,\ell_m ). \]
When we consider polytopes of the same type $(I; \ell_1 ,\dots ,\ell_m )$ 
simultaneously, 
we denote one as $\Delta_m(I; \ell_1 ,\dots ,\ell_m )$
and the others as 
$\Delta'_m(I; \ell_1 ,\dots ,\ell_m )$, 
$\Delta''_m(I; \ell_1 ,\dots ,\ell_m )$ and so on.

\begin{lem}\label{unique}\it
The following holds up to $\mathrm{Aff}(\mathbb{Z}^2)$-equivalence: \\
(1)\;
A triangle $\Delta_{3}(3;1,1,1)$ 
is either $\Delta_{\I}$ or $\Delta_{\II}$. \\
(2)\;
A triangle 
$\Delta_{3}(2;2,1,1)$ is $\Delta_{\III}$. \\
(3)\;
A triangle 
$\Delta_{3}(1;2,1,1)$ is $\Delta_{\IV}$. \\
(4)\;
A triangle 
$\Delta_3(0;4,1,1)$ is $\Delta_{\V}$. \\
(5)\;
A parallelogram 
$\Delta^{\mathrm{par}}_{4}(2;1,1)$ is $\Delta_{\VI}$. \\
(6)\;
A pentagon 
$\Delta_{5}(1;1,1,1,1,1)$ is $\Delta_{\VII}$. \\
(7)\;
A non-parallel quadrangle 
$\Delta_{4}(2;1,1,1,1)$ 
is
equivalent to one of the following polytopes: 
\[
\Delta_{\VIII}, \;\;
\Delta_{\IX}, \;\;
\mathrm{Conv}\{ (1,0),(0,1),(2,1),(1,3) \}.
\]
\end{lem}

\begin{proof}
(1)\;
We can take 
$A \in \mathrm{Aff}(\mathbb{Z}^2)$ 
which maps 
$\Delta_{3}(3;1,1,1)$ 
to 
\[
\Hat{\Delta}_n
:=\mathrm{Conv}\{(0,q),(n,0),(n+1,0)\}
\]
for some $q, n$. 
By Pick's formula, we obtain $q=7$. 
We remark that,  
$\Hat{\Delta}_n$ 
and
$\Hat{\Delta}_{n+7}$ 
are equivalent by 
\begin{equation}\tag{***}
\begin{pmatrix}
1 & 1 \\
0 & 1 \label{eq:1101}
\end{pmatrix}. 
\end{equation}
Moreover we do not have to discuss the cases $n=0$ and $n=6$
since they have an edge of lattice length more than $1$.

We get the isomorphisms 
\[
\Hat{\Delta}_1 \simeq 
\Hat{\Delta}_5,\ \ \ 
\Hat{\Delta}_2 \simeq 
\Hat{\Delta}_4 
\]
by the reflection, and 
$\Hat{\Delta}_1 \simeq \Hat{\Delta}_3$
by 
\[
\begin{pmatrix}
 3 &  1 \\
-7 & -2 
\end{pmatrix}.
\]
Because of the configuration of 
interior lattice points, 
we can show that 
$\Hat{\Delta}_1 = \Delta_{\I}$ and $\Hat{\Delta}_2=\Delta_{\II}$ 
are not isomorphic. \\
(2)\;
For any 
$\Delta_{3}(2;2,1,1)$, 
there exists 
$A \in \mathrm{Aff}(\mathbb{Z}^2)$ 
such that $\Delta_{3}(2;2,1,1)$ 
maps to 
\[\mathrm{Conv} \{ (p,0),(p+2,0),(0,q) \}\]
for some $p,q \in \mathbb{N}$. 
Then we have $q=3$ by Pick's formula, 
and we may assume $p=0,1,2$ by the isomorphism (\ref{eq:1101}). 
But the cases $p=0,1$ do not satisfy the conditions of lattice length. 
Hence we get $p=2$.
This triangle is equivalent to 
$\Delta_{\III}$. 

The claims (3), (4), (5) and (6) can be proved 
by the same method. \\
(6)\; 
We can split $P:=\Delta_4(2;1,1,1,1)$ into two triangles 
which satisfies one of the following: \\
\textbullet \;
$\Delta_3(1;2,1,1)$ and $\Delta_3(0;2,1,1)$ 
such that their intersection is a segment of length~$2$, \\
\textbullet \;
$\Delta_3(2;1,1,1)$ and $\Delta_3(0;1,1,1)$ 
such that their intersection is a segment of length~$1$, \\
\textbullet \;
$\Delta_3(0;3,1,1)$ and $\Delta'_3(0;3,1,1)$ 
such that their intersection is a segment of length~$3$, \\
\textbullet \;
$\Delta_3(1;1,1,1)$ and $\Delta'_3(1;1,1,1)$ 
such that their intersection is a segment of length~$1$. 

In the first case, 
$\Delta_3(1;2,1,1)$ 
is uniquely determined as $\mathrm{Conv}\{ (0,0),(2,0), (1,2)\}$, 
so $P$ has two descriptions
\[
\hat{P}_1
:=\mathrm{Conv}\{ (0,0),(2,0),(1,2),(0,-1) \},\;\;
\hat{P}_2
:=\mathrm{Conv}\{ (0,0),(2,0),(1,2),(1,-1) \}. 
\]

In the second case, by \cite[Lemma 4.1]{S}, any triangle 
$\Delta_3(2;1,1,1)$ is isomorphic to 
\[
Q:=\mathrm{Conv}\{ (0,0), (3,2), (2,3) \}.
\]
We denote the other triangle, which is $\Delta_3(0;1,1,1)$,  
by $R$. 
We can easily check that $Q$ is equivalent to
\[
Q_1:=\mathrm{Conv}\{ (0,1), (0,2), (0,5) \}, \;\; 
Q_2:=\mathrm{Conv}\{ (0,2), (0,3), (0,5) \}. 
\]
If the intersection of $Q$ with $R$ 
is 
$\mathrm{Conv}\{(0,0), (3,2)\} \subset Q$ 
or 
$\mathrm{Conv}\{(0,0), (2,3)\} \subset Q$, 
then we can assume that 
the intersection is the bottom edge of $Q_1$. 
Similarly, if the intersection is 
$\mathrm{Conv}\{(2,3), (3,2)\} \subset Q$, 
then we can assume that $R$ shares the bottom edge of $Q_2$. 
Thus, the polytope $P$ is equivalent to either 
\[
\hat{P}_3
:=\mathrm{Conv}\{ (1,0),(2,0),(0,5),(2,-1) \}
\;\; \text{or} \;\;
\hat{P}_4
:=\mathrm{Conv}\{ (2,0),(3,0),(0,5),(3,-1) \}. 
\]

In the third and fourth cases, 
we obtain the following polytopes in the same way as above: 
\[
\hat{P}_5
:=\mathrm{Conv}\{ (0,0),(0,1),(1,-1),(3,0) \}, \;\; 
\hat{P}_6
:=\mathrm{Conv}\{ (0,0),(0,1),(2,-1),(3,0) \}. 
\]

Between the polytopes $\hat{P}_1,\dots,\hat{P}_6$, 
we have the following isomorphisms: 
\[
\hat{P}_1 \simeq \hat{P}_3 
\;\; \text{by} \;\; 
\begin{pmatrix}
-1 & 0 \\
 3 & -1 \\
\end{pmatrix}, \;\;
\hat{P}_5 \simeq \hat{P}_4 
\;\; \text{by} \;\; 
\begin{pmatrix}
-1 & -1 \\
 2 &  1 \\
\end{pmatrix}, \;\;
\hat{P}_6 \simeq \hat{P}_2 
\;\; \text{by} \;\; 
\begin{pmatrix}
 0 & -1 \\
 1 &  1 \\
\end{pmatrix}.  
\]
Notice that, the polytope $\hat{P}_2$
is the translation of 
$\mathrm{Conv}\{(1,0), (0,1), (2,1), (1,3)\}$. 
Also, the polytopes 
$\hat{P}_3$ and $\hat{P}_4$ 
are equivalent to 
$\Delta_{\IX}$ 
and 
$\Delta_{\VIII}$ 
by 
\[ 
\begin{pmatrix}
1 & 1 \\
-1 & 0 \\
\end{pmatrix}
:\Z^2 \to \Z^2, 
\]
respectively. 

Furthermore, 
by the configuration of interior lattice points and vertices, 
we obtain 
$\Delta_{\VIII} \not\simeq \Delta_{\IX}$, 
$\Delta_{\IX} \not\simeq \mathrm{Conv}\{(1,0), (0,1), (2,1), (1,3)\}$, 
and 
$\mathrm{Conv}\{(1,0), (0,1), (2,1), (1,3)\} \not\simeq \Delta_{\VIII}$. 
\end{proof}

\begin{lem}\label{expol}\it
A quadrangle $\Delta_4(1;2,1,1,1)$ is $\Delta_{\E}$. 
\end{lem}

\begin{proof}
We can split 
$P=\Delta_4(1;2,1,1,1)$
into two polytopes $Q$, $R$ which are either 
\begin{itemize}
\item[(3-1)]
$Q=\Delta_3(0;1,1,1)$, $R=\Delta_4(1;1,1,1,1)$ 
and these polytopes share an edge of length $1$, 
\item[(3-2)] 
$Q=\Delta_3(0;2,1,1)$, $R=\Delta_4(0;2,1,1,1)$ 
and these polytopes share the edge of length $2$, 
\item[(3-3)] 
$Q=\Delta_3(1;1,1,1)$, $R=\Delta_4(0;1,1,1,1)$ 
and these polytopes share an edge of length $1$, 
\item[(3-4)] 
$Q=\Delta_3(1;2,1,1)$, $R=\Delta_3(0;1,1,1)$ 
and these polytopes share an edge of length $1$, 
\item[(3-5)] 
$Q=\Delta_3(0;2,2,1)$, $R=\Delta_3(0;2,1,1)$ 
and these polytopes share an edge of length $2$, or 
\item[(3-6)] 
$Q=\Delta_3(0;2,1,1)$, $R=\Delta_3(1;1,1,1)$ 
and these polytopes share an edge of length $1$. 
\end{itemize}
Among them, case (3-5) can not occur by Lemma \ref{polyt1}. \\
(3-1)\;
If $R$ is a parallelogram, then we can assume that $R$ is 
\[
\mathrm{Conv}\{ (1,0), (2,0), (0,2), (1,2) \}
\]
and the common edge of $R$ with $Q$ is its bottom edge. 
Hence, we get 
\[
Q=\mathrm{Conv}\{ (1,0), (2,0), (2,-1) \},
\] 
by Pick's formula, 
but their union does not satisfy the condition of $P$. 

If $R$ is not a parallelogram, then we can assume that $R$ is 
\[
\mathrm{Conv}\{ (0,0), (1,0), (0,1), (2,2) \}
\]
and the common edge with $Q$ is either 
\[
\mathrm{Conv}\{ (0,0), (1,0) \}
\;\; \text{or} \;\;
\mathrm{Conv}\{ (1,0), (2,2) \}. 
\] 
In the former case, $Q$ is uniquely determined as 
\[
\mathrm{Conv}\{ (0,0), (1,0), (0,-1)  \} 
\]
and the union 
$Q \cup R = \mathrm{Conv}\{ (0,-1), (1,0), (0,1), (2,2) \}$
is isomorphic to $P$. 
In the latter case, we can assume that $R$ is 
\[
\mathrm{Conv} \{ (1,0), (2,0), (0,2), (0,3) \}
\]
and the common edge is the bottom edge. 
Then $Q$ must be 
\[
\mathrm{Conv}\{ (1,0), (2,0), (2,1) \}, 
\]
but the union $Q \cup R$ does not satisfy the condition of $P$. \\
(3-2)\;
We can assume that $R$ is 
\[
\mathrm{Conv}\{ (0,0), (2,0), (0,1), (1,1) \}
\]
and the common edge is the bottom edge. 
Then $Q$ must be either 
\[
\mathrm{Conv}
\{ (0,0), (2,0), (0,-1) \}
\;\; \text{or}\;\;
\mathrm{Conv}
\{ (0,0), (2,0), (3,-1) \}. 
\]
In both cases, the union $Q \cup R$ are isomorphic to $P$. \\
(3-3)\;
We can assume that $R$ is 
\[ \mathrm{Conv}\{ (0,0),(1,0),(0,1),(1,1) \}, \]
but any union with $Q$ does not satisfy the condition of $P$. \\
(3-4)\;
We can assume that $Q$ is 
\[
\mathrm{Conv}\{ (0,0), (1,0), (-2,4) \}
\]
and the common edge is its bottom edge. 
Then $R$ must be 
\[
\mathrm{Conv}
\{ (0,0), (1,0), (1,-1) \}. 
\]
Their union $R \cup Q$ is isomorphic to $P$. \\
(3-6)\;
We assume that $R$ is 
\[
\mathrm{Conv}\{ (1,0), (2,0), (0,3) \}
\]
and the common edge is its bottom edge. 
Then $Q$ must be either 
\[
\mathrm{Conv}
\{ (1,0), (2,0), (2,-2) \}
\;\; \text{or}\;\;
\mathrm{Conv}
\{ (1,0), (2,0), (3,-2) \}. 
\]
In both cases, the union $Q \cup R$ is isomorphic to $P$. 
\end{proof}

\subsection{Existence of $1$-tacnodal curves for $\Delta_{\I}, \dots 
\Delta_{\IX}$}
\label{existtac}

For a polytope $P$, 
we set 
\[
\mathcal{F}(P) 
:=\{ f \in \C[x,y];N_f=P \}.
\]
We denote the plane curve defined by 
$f \in \mathcal{F}(P)$ in $X(P)$ 
as $ V_f$. 
We remark that 
$V_f$ is a member of $|D(P)|$. 
We consider the following two conditions: 
\begin{itemize}
\item[(S1)] 
$V_f \subset X(P)$ 
is a 1-tacnodal curve whose singular point is contained in 
the maximal torus of $X(P)$, 
\item[(S2)]
$V_f$ intersects the toric boundary $X(\partial P)$ transversally. 
\end{itemize}

In the rest of this section, 
except cases ($\III$), ($\IV$), and ($\V$), 
we only consider polytopes whose edges are only of length one. 
Hence the condition (S2) is automatically satisfied 
except the three cases.

\begin{lem}\label{tacpiece}\it
For each $i=\I, \II$ 
and given coefficients 
$c_{ij}$ on the vertices $(i,j) \in V(P)$, 
there is a polynomial 
$f \in \mathcal{F}(\Delta_i)$
which has the fixed coefficients on the vertices 
and satisfies the conditions (S1), (S2). 
Furthermore, 
there is no polynomial 
$f\in \mathcal{F}(\Delta_i)$ 
that defines a curve with 
more complicated singularity than $A_3$, 
i.e., 
the curve does not have an isolated singularity 
whose Milnor number is more than $3$.  
\end{lem}

\begin{proof}
(I)\;
We first show that we can assume that 
the coefficients on the vertices of 
$\Delta_{\I}$ are $1$.  
We transform the polynomial 
\[
f=c_{10}x+c_{20}x^2+Axy +Bxy^2 +Cxy^3 +c_{07}y^7
\in \mathcal{F}(\Delta_{\I})
\]
by substituting $x=X^{-1}, y=Y$ and multiplying $X^2$. 
Then we get a new polynomial 
\[
\tilde{f}:=
c_{20}+c_{10}X+AXY+BXY^2+CXY^3+c_{07}X^2Y^7. 
\]
By multiplying suitable constants to the variables 
and the whole polynomial, 
we can assume that 
$c_{20}=c_{10}=c_{07}=1$. 
Transforming $\tilde{f}$ by $x=X^{-1}, y=Y$ again, 
we get 
\[
x+x^2+A'xy+B'xy^2+C'xy^3+y^7. 
\]
We re-denote this polynomial by $f$.

For a polynomial 
\[ 
f=x+x^2+Axy+Bxy^2+Cxy^3+y^7 \in \mathcal{F}(\Delta_{\I}), 
\]
we apply Lemma \ref{tacnode} and eliminate the variables 
by the system 
$f=f_x=f_y=\mathrm{Hess}(f)=K(f)=0$. 
First, 
by $f=0$, we can get $A$ as 
\[
A=-\frac{x+x^2+Bxy^2+Cxy^3+y^7}{xy}. 
\]
Therefore the system is reduced as 
\[
\begin{cases}
&\text{(1)}\;\; x^2-y^7                  =0, \\
&\text{(2)}\;\; -x-x^2+Bxy^2+2Cxy^3+6y^7 =0, \\
&\text{(3)}\:\; \text{ substituting $A$ for }\mathrm{Hess}(f)=0,  \\
&\text{(4)}\:\; \text{ substituting $A$ for }K(f)=0. 
\end{cases}
\]
Secondly, 
by equation (2), we can get $B$ as 
\[
B=\frac{x+x^2-2Cxy^3-6y^7}{xy^2}. 
\]
Then the system is reduced as
\[
\begin{cases}
&\text{(1')}\;\; 
x^2-y^7=0, \\
&\text{(3')}\;\; 
4x^3+4x^4+4Cx^3y^3+60x^2y^7-49y^{14}=0, \\
&\text{(4')}\;\;
2Cx^3+7xy^4+77x^2y^4+7Cxy^7-42y^{11}=0. 
\end{cases}
\]
Thirdly, by equation (3'), we can get $C$ as 
\[
C=\frac{
-4x^3-4x^4-60x^2y^7+49y^{14}}{4x^3y^3}.  
\]
Then the system is reduced as 
\[
\begin{cases}
& x^2-y^7=0, \\
& 
8x^5+8x^6-160x^4y^7+490x^2y^{14}-343y^{21}=0. 
\end{cases}
\]
Hence we obtain $x=8/5$ and the equation 
\[ \label{sol1} \tag{****}
y^7-(8/5)^2=0. 
\]

Next, we check that the above $f$ satisfies 
the condition (S1). 
Let $y_0, y_1,\ldots,y_6$ be the solutions of equation \eqref{sol1} 
and, for each $i=1, \ldots, 6$, 
let $f^{(i)}$ denote the polynomial $f$ with the solution $y=y_i$. 
By the above calculation, 
the curve $V_{f^{(i)}}$ defined by $f^{(i)}$ 
has a tacnode at $(8/5, y_i)\in (\C^*)^2$. 
Notice that the coefficients $A, B$ and $C$ of $f^{(i)}$ 
are determined by $x=8/5$ and $y=y_i$. 
Let $(s,t)$ be a singular point of $f^{(i)}$ on $V_{f^{(i)}}$. 
Solving $f^{(i)}_s=0$, we obtain $s=s_0(t,y_0)$. 
Set 
\[
f_1(t,y_0):=f^{(i)}(s_0(t,y_0),t), \;\;\;
f_2(t,y_0):=f^{(i)}_t(s_0(t,y_0),t). 
\]
Eliminating $y_0$ from $f_1, f_2$ by $y_0^7-(8/5)^2=0$, 
we obtain two equations with variable $t$. 
We can check that 
their greatest common divisor is $t^7-(5/8)^2$. 
Thus, the singularities of $f^{(i)}$ are only tacnodes. 

The coefficient $A$ of $f^{(i)}$ 
depends only on the solution $y_0$ of \eqref{sol1} 
and we can check directly that the coefficients $A$ for $y=y_i$
and $y=y_j$ are different if $i\ne j$. 
That is, the defining polynomials $f^{(i)}$
and $f^{(j)}$ are different for $i\ne j$. 
Therefore each $f^{(i)}$ satisfies the condition (S1).

\vspace{2mm}
\noindent
(II)\; 
For the polynomial 
\[ 
f=x^2+x^3+Ax^2y+Bx^2y^2+Cxy^4+y^7 \in \mathcal{F}(\Delta_{\II}), 
\]
we apply Lemma \ref{tacnode} 
and eliminate the variables by the system 
$f=f_x=f_y=\mathrm{Hess}(f)=K(f)=0$. 
First, 
by $f=0$, we can get $A$ as 
\[
A=-\frac{x^2+x^3+Bx^2y^2+Cxy^4+y^7}{x^2y}. 
\]
Therefore the system is reduced as 
\[
\begin{cases}
&\text{(1)}\;\; x^3-Cxy^4-2y^7                =0, \\
&\text{(2)}\;\; -x^2-x^3+Bx^2y^2+3Cxy^4+6y^7  =0, \\
&\text{(3)}\:\; \text{ substituting $A$ for }\mathrm{Hess}(f)=0,  \\
&\text{(4)}\:\; \text{ substituting $A$ for }K(f)=0. 
\end{cases}
\]
Secondly, 
by equation (2), we can get $B$ as 
\[
B=\frac{x^2+x^3-3Cxy^4-6y^7}{x^2y^2}. 
\]
Then the system is reduced as
\[
\begin{cases}
&\text{(1')}\;\;  
x^3-Cxy^4-2y^7 =0, \\
&\text{(3')}\;\; 
8x^5+8x^6-4Cx^3y^4+20Cx^4y^4-4x^2y^7
+116x^3y^7-28C^2x^2y^8-184Cxy^{11}-256y^{14}=0, \\
&\text{(4')}\;\;
\text{ substituting $B$ for (4) }=0. 
\end{cases}
\]
Thirdly, by equation (1'), we can get $C$ as 
\[
C=\frac{x^3-2y^7}{xy^4}.  
\]
Then the system is reduced as 
\[
\begin{cases}
& 
x^3+y^7+xy^7 =0, \\
& 
4x^9+14x^6y^7+5x^7y^7+16x^3y^{14}+11x^4y^{14}
+6y^{21}+7xy^{21}=0. 
\end{cases}
\]
By direct computation, we can see that the solution of the above system is 
\[\tag{*****}\label{sol2}
(x,y)=(y_0^7,y_0),
\]
where $y_0$ is a solution of $y^{14}+y^7+1=0$. 

Next, we check that the above $f$ satisfies 
the condition (S1). 
Notice that the curve $V_f$ defined by $f$ 
has a tacnode at 
$(y_0^7, y_0) \in (\C^*)^2$, 
where $y_0$ is a solution of \eqref{sol2}. 
Let $(s,t) \in (\C^*)^2$
be a singular point of $V_f$. 
Then, we can easily check that 
the system 
$f(s,t)=f_x(s,t)=f_y(s,t)=y_0^{14}+y_0^7+1=0$ 
implies 
$t=y$. 
After substituting $y=t$ for $f(s,t), f_x(s,t), f_y(s,t)$, 
we obtain 
$s-y_0^7$ 
as their greatest common divisor. 
That is, the singularities of $V_f$ are only tacnodes. 
Moreover, we can easily check that for two different solutions $y_0$ and 
$y_0'$ of $y^{14}+y^7+1=0$, 
the triples $(A,B,C)$ of the coefficients of the polynomial $f$, 
which are determined by $y_0$ and $y_0'$, are different.
Therefore, for each solution of $y^{14}+y^7+1=0$, 
the polynomial $f$ satisfies the condition (S1).
\end{proof}

\begin{lem}\label{tacpiece2}\it
For each $i=\VI, \VII, \VIII, \IX$, 
and given coefficients $c_{ij}$ on the vertices 
$(i,j) \in V(P)$, 
there is a polynomial 
$f \in \mathcal{F}(\Delta_i)$
which has the fixed coefficients on the vertices 
and satisfies (S1) and (S2) 
if and only if 
\begin{align*}
\bullet \; &c_{03}c_{20}=64c_{10}c_{13}, 
\;\;\; \text{if} \;\;\; i=\VI, \\
\bullet \; 
&c_{21}c_{00}^2=-4c_{01}c_{10}^2, \;\; \text{and}\;\; 
c_{12}c_{00}^2=-4c_{10}c_{01}^2,
\;\;\; \text{if}\;\;\;  i=\VII, \\
\bullet \; &8^6c_{33}c_{00}^5=5^5c_{10}^3c_{01}^3,  
\;\;\; \text{if}\;\;\;  i=\VIII, \\
\bullet \; &256c_{42}c_{00}^5=(41+38\sqrt{-1})c_{10}^4c_{01}^2, 
\;\; \text{or} \;\; 
256c_{42}c_{00}^5=(41-38\sqrt{-1})c_{10}^4c_{01}^2,
\;\;\; \text{if}\;\;\;  i=\IX. 
\end{align*}

Furthermore, 
there is no polynomial 
$f\in \mathcal{F}(\Delta_i)$ 
that defines a curve with 
more complicated singularity than $A_3$. 
\end{lem}

\begin{proof}
(VI)\;
We transform the polynomial 
\[
f:=c_{10}x+c_{20}x^2+Axy+Bxy^2+c_{03}y^3+c_{13}xy^3
\in \mathcal{F}(\Delta_{\VI})
\]
by substituting $x=X^{-1}, y=Y$ and multiplying $X^2$. 
Then we get the new polynomial 
\[
\tilde{f}:=
c_{10}X+c_{20}+A'XY+B'XY^2+c_{03}X^2Y^3+c_{13}XY^3. 
\]
By multiplying suitable constants to the variables 
and the whole polynomial, 
we can rewrite $\tilde{f}$ as 
\[
1+X+A''XY+B''XY^2+XY^3+CX^2Y^3, 
\]
where 
\[
C=\frac{c_{03}c_{20}}{c_{10}c_{13}}. 
\]  

For the polynomial 
\[
1+x+Axy+Bxy^2+xy^3+Cx^2y^3
\]
we apply Lemma \ref{tacnode} 
and eliminate the variables by the system 
$f=f_x=f_y=\mathrm{Hess}(f)=K(f)=0$. 
First, 
by $f=0$, we can get $A$ as 
\[
A=-\frac{1+x+Bxy^2+xy^3+Cx^2y^3}{xy}. 
\]
Therefore the system is reduced as 
\[
\begin{cases}
&\text{(1)}\;\; -1+Cx^2y^3                 =0, \\
&\text{(2)}\;\; -1-x+Bxy^2+2xy^3+2Cx^2y^3  =0, \\
&\text{(3)}\:\; \text{ substituting $A$ for }\mathrm{Hess}(f)=0,  \\
&\text{(4)}\:\; \text{ substituting $A$ for }K(f)=0. 
\end{cases}
\]
Secondly, 
by equation (1), we can get $C$ as 
\[
C=\frac{1}{x^2y^3}. 
\]
Then the system is reduced as
\[
\begin{cases}
&\text{(2')}\;\; 
1-x+Bxy^2+2xy^3 =0, \\
&\text{(3')}\;\; 
-4 +8x -x^2 -4Bxy^2 
+2Bx^2y^2 -4xy^3 +4x^2y^3 
-B^2x^2y^4 -4Bx^2y^5 -4x^2y^6 =0, \\
&\text{(4')}\;\;
48 -144x +36x^2 
+48 Bxy^2 +48xy^3 -48Bx^2y^2 
-72x^2y^3 +12B^2x^2y^4 +24Bx^2y^5 =0. 
\end{cases}
\]
Thirdly, by equation (2'), we can get $B$ as 
\[
B=-\frac{1-x+2xy^3}{xy^2}.  
\]
Then the system is reduced as 
\[
\begin{cases}
& 4x+4xy^3-1  =0, \\
& 6x+2xy^3-1  =0. 
\end{cases}
\]
The solution of the above system is 
\[\tag{******}\label{sol3}
(x,y)=(1/8,y_0),
\]
where $y_0$ is a solution of $y^3=1$. 
Then we obtain 
\[
A=-9/y_0, \;\;
B=-9/y_0^2, \;\;
C=1/x^2y^3=64. 
\]

Next, we check that the above $f$ satisfies the condition (S1). 
Notice that the curve $V_f$ defined by $f$ 
has a tacnode at $(1/8, y_0) \in (\C^*)^2$, 
where $y_0$ is a solution of \eqref{sol3}. 
Let $(s,t) \in (\C^*)^2$ be a singular point of $V_f$. 
Then, we obtain 
$t^3-1=0$ and $s=1/8$
from the system 
$f(s,t)=f_x(s,t)=f_y(s,t)=0$ 
and the equation 
$y_0^3-1=0$. 
That is, the singularities of $V_f$ are only tacnodes. 
Moreover, we can easily check that for two different solutions 
$y_0$ and $y_0'$ of $y^3-1=0$, 
the triples $(A,B,C)$ of the coefficients of the polynomial $f$, 
which are determined by $y_0$ and $y_0'$, are different.
Therefore, for each solution of $y^3-1=0$, 
the polynomial $f$ satisfies the condition (S1). 

\vspace{2mm}
\noindent
(VII)\;
We can rewrite the polynomial 
\[
f=c_{00}+c_{10}x+c_{01}y+Axy+c_{21}x^2y+c_{12}xy^2
\in \mathcal{F}(\Delta_{\VII})
\]
as 
\[
f=1+x+y+Axy+Bx^2y+Cxy^2
\]
by the same manner as above, where 
\[
B=\frac{c_{21}c_{00}^2}{c_{01}c_{10}^2}, \;\;\;
C=\frac{c_{12}c_{00}^2}{c_{10}c_{01}^2}. 
\]
For the polynomial 
\[
f=1+x+y+Axy+Bx^2y+Cxy^2, 
\]
we apply Lemma \ref{tacnode} 
and eliminate the variables by the system 
$f=f_x=f_y=\mathrm{Hess}(f)=K(f)=0$. 
First, 
by $f=0$, we can get $A$ as 
\[
A=-\frac{1+x+y+Bx^2y+Cxy^2}{xy}. 
\]
Therefore the system is reduced as 
\[
\begin{cases}
&\text{(1)}\;\; -1-y+Bx^2y         =0, \\
&\text{(2)}\;\; -1-x+Cxy^2         =0, \\
&\text{(3)}\:\; \text{ substituting $A$ for }\mathrm{Hess}(f)=0,  \\
&\text{(4)}\:\; \text{ substituting $A$ for }K(f)=0. 
\end{cases}
\]
Secondly, 
by equations (1) and (2), we can get $B$ and $C$ as 
\[
B=\frac{1+y}{x^2y},\quad
C=\frac{1+x}{xy^2}, 
\]
respectively. 
Then the system is reduced as
\[
\begin{cases}
&\text{(3')}\;\;  
3+4x+4y+4xy      =0, \\
&\text{(4')}\;\; 
(1+y)^2(1+2x)    =0. 
\end{cases}
\]
The solution of the above system is 
\[
(x,y)=(-1/2,-1/2),  
\]
and we obtain 
\[
A=B=C=-4. 
\]

Next, we check that the above $f$ satisfies the condition (S1). 
Notice that the curve $V_f$ defined by $f$ 
has a tacnode at $(-1/2, -1/2) \in (\C^*)^2$. 
Let $(s,t) \in (\C^*)^2$ be a singular point of $V_f$. 
Then, we can solve 
$f(s,t)=f_x(s,t)=f_y(s,t)=0$, 
and the solution is $(s,t)=(-1/2, -1/2)$. 
That is, the singularity of $f$ is only one point and is a tacnode. 
Therefore the $f$ satisfies the condition (S1). 

\vspace{2mm}
\noindent
(VIII)\;
We can rewrite the polynomial 
\[
f=c_{00}+c_{10}x+c_{01}y +Axy+Bx^2y^2+c_{33}x^3y^3
\in \mathcal{F}(\Delta_{\VIII})
\]
as
\[
f=1+x+y+Axy+Bx^2y^2+Cx^3y^3 
\]
by the same manner as above, where 
\[
C=\frac{c_{33}c_{00}^5}{c_{10}^3c_{01}^3}. 
\]

For the polynomial 
\[
f=1+x+y+Axy+Bx^2y^2+Cx^3y^3, 
\]
we apply Lemma \ref{tacnode} 
and eliminate the variables by the system 
$f=f_x=f_y=\mathrm{Hess}(f)=K(f)=0$. 
First, 
by $f=0$, we can get $A$ as 
\[
A=-\frac{1+x+y+Bx^2y^2+Cx^3y^3}{xy}. 
\]
Therefore the system is reduced as 
\[
\begin{cases}
&\text{(1)}\;\; -1-y+Bx^2y^2+2Cx^3y^3         =0, \\
&\text{(2)}\;\; -1-x+Bx^2y^2+2Cx^3y^3         =0, \\
&\text{(3)}\:\; \text{ substituting $A$ for }\mathrm{Hess}(f)=0,  \\
&\text{(4)}\:\; \text{ substituting $A$ for }K(f)=0. 
\end{cases}
\]
Secondly, 
by equation (1), we can get $B$ as 
\[
B=\frac{1+y-2Cx^3y^3}{x^2y^2}. 
\]
Then the system is reduced as
\[
\begin{cases}
&\text{(2')}\;\;  x-y            =0, \\
&\text{(3')}\;\; 4-x+4y+4Cx^3y^3=0, \\
&\text{(4')}\;\;
\text{ substituting $B$ for (4) }=0. 
\end{cases}
\]
Thirdly, by equation (3'), we can get $C$ as 
\[
C=\frac{-4+x-4y}{4x^3y^3}.  
\]
Then the system is reduced as 
\[
\begin{cases}
& x-y         =0, \\
& -8 +3x-8y     =0. 
\end{cases}
\]
The solution of the above system is 
\[
(x,y)=(-8/5,-8/5), 
\] 
and we also obtain 
\[
A=75/64, \;\;
B=-5^4/2^{12}, \;\;
C=5^5/8^6. 
\]

Next, we check that the above $f$ satisfies the condition (S1). 
Notice that the curve $V_f$ defined by $f$ 
has a tacnode at $(-8/5, -8/5) \in (\C^*)^2$. 
Let $(s,t) \in (\C^*)^2$ be a singular point of $V_f$. 
Then, we can solve 
$f(s,t)=f_x(s,t)=f_y(s,t)=0$, 
and the solution is $(s,t)=(-8/5, -8/5)$. 
That is, the singularity of $f$ is only one point and is a tacnode. 
Therefore the $f$ satisfies the condition (S1). 

\vspace{2mm}
\noindent
(IX)\;
We can rewrite the polynomial 
\[
f=c_{00}+c_{10}x+c_{01}y +Axy+Bx^2y+c_{42}x^4y^2
\in \mathcal{F}(\Delta_{\IX})
\]
as
\[
f=1+x+y+Axy+Bx^2y+Cx^4y^2 
\]
by the same manner as above, where 
\[
C=\frac{c_{42}c_{00}^5}{c_{10}^4c_{01}^2}. 
\]

For the polynomial 
\[
f=1+x+y+Axy+Bx^2y+Cx^4y^2, 
\]
we apply Lemma \ref{tacnode} 
and eliminate the variables by the system 
$f=f_x=f_y=\mathrm{Hess}(f)=K(f)=0$. 
First, 
by $f=0$, we can get $A$ as 
\[
A=-\frac{1+x+y+Bx^2y+Cx^4y^2}{xy}. 
\]
Therefore the system is reduced as 
\[
\begin{cases}
&\text{(1)}\;\; -1-y+Bx^2y +3Cx^4y^2        =0, \\
&\text{(2)}\;\; -1-x+Cx^4y^2         =0, \\
&\text{(3)}\:\; \text{ substituting $A$ for }\mathrm{Hess}(f)=0,  \\
&\text{(4)}\:\; \text{ substituting $A$ for }K(f)=0. 
\end{cases}
\]
Secondly, 
by equations (1), we can get $B$ as 
\[
B=\frac{1+y-3Cx^4y^2}{x^2y}. 
\] 
Then the system is reduced as
\[
\begin{cases}
&\text{(2')}\;\;  
-1-x+Cx^4y^2     =0, \\
&\text{(3')}\;\; 
1 -4Cx^2y^2 -8Cx^3y^2
-4Cx^2y^3 +4C^2x^6y^4   =0, \\
&\text{(4')}
\text{ substituting $B$ for (4) }=0.
\end{cases}
\]
Thirdly, by equations (2'), we can get $C$ as 
\[
C=\frac{1+x}{x^4y^2}. 
\]
Then the system is reduced as 
\[
\begin{cases}
&
4x +4y +3x^2 +4xy =0, \\
&
(4+3x)
(
16x +8y
+24x^2 +22xy +4y^2
+9x^3 +12x^2y +5xy^2
) =0.
\end{cases}
\]
The solutions of the above system are 
\[
(x_0,y_0)
=
\Bigl(
-\frac{6}{5}+\frac{2}{5}\sqrt{-1}, \frac{2}{5}-\frac{4}{5}\sqrt{-1}
\Bigr), 
\;\;
(x_1,y_1)
=
\Bigl(
-\frac{6}{5}-\frac{2}{5}\sqrt{-1}, \frac{2}{5}+\frac{4}{5}\sqrt{-1}
\Bigr), 
\]
and we obtain 
\[
C=-\frac{41}{256}+\frac{19}{128}\sqrt{-1}
\;\;\; \text{if} \;\;\; 
x= x_0, \;\;\;
C=\frac{41}{256}+\frac{19}{128}\sqrt{-1}
\;\;\; \text{if}\;\;\;
x= x_1. 
\]

Next, we check that the above $f$ satisfies the condition (S1). 
Notice that the curve $V_f$ defined by $f$ 
has a tacnode at $(x_0, y_0) \in (\C^*)^2$. 
Let $(s,t) \in (\C^*)^2$ be a singular point of $V_f$. 
Then, we can solve 
$f(s,t)=f_x(s,t)=f_y(s,t)=0$, 
and the solution is $(s,t)=(x_0, y_0)$. 
That is, the singularity of $f$ is only one point and is a tacnode. 
Therefore the $f$ satisfies the condition (S1). 
Also, 
we can check the condition (S1) for $(x_1,y_1)$ 
by the same manner. 
\end{proof}

\begin{lem}\label{singedge}\it
For each $i=\III, \IV, \V$ 
and given coefficients 
$c_{ij}$ on the vertices $(i,j) \in V(P)$, 
there is a polynomial 
$f \in \mathcal{F}(\Delta_i)$
which has the fixed coefficients on the vertices 
such that $f$ defines a curve which has 
\begin{itemize}
\item[(III)]
an $A_2$-singularity on the toric divisor  
corresponding to the edge of length $2$,  
\item[(IV)]
an $A_1$-singularity on the toric divisor 
corresponding to the edge of length $2$, 
\item[(V)]
an intersection with the toric divisor 
corresponding to the edge of length $4$
whose multiplicity is $4$. 
\end{itemize}
\end{lem}
\begin{proof}
(III)\; 
We set 
\[
f:=1+Ax+x^2+Bxy+Cxy^2+xy^3 \in \mathcal{F}(\Delta_{\III}). 
\]
Let 
$\sigma \subset \Delta_{\III}$ 
be the edge of length $2$. 
The intersection point of $X(\sigma)$ 
and the curve defined by $f$ 
is an $A_2$-singularity and this implies $A=\pm 2$. 

We assume $A=2$ and the singularity is at $(-1,0)$. 
For 
$f=(1+x)^2+Bxy+Cxy^2+y^3$, 
the solution of
$f(-1,0)=f_x(-1,0)=f_y(-1,0)=\mathrm{Hess}(f)(-1,0)=0$ 
is $B=C=0$. 
Therefore we obtain the polynomial 
$f:=1+2x+x^2+xy^3 \in \mathcal{F}(\Delta_{\III})$. 

\vspace{2mm}
\noindent
(IV)\;
We set 
\[
f:= 1+Ax+x^2+Bxy+xy^2 
\in \mathcal{F}(\Delta_{\IV}). 
\]
Let 
$\sigma \subset \Delta_{\IV}$ 
be the edge of length $2$. 
The intersection point of $X(\sigma)$ 
and the curve defined by $f$ 
is an $A_1$-singularity and this implies $A=\pm 2$. 

We assume $A=2$ and the singularity is at $(-1,0)$. 
For 
$f=(1+x)^2+Bxy+xy^2$, 
the solution of
$f(-1,0)=f_x(-1,0)=f_y(-1,0)=0$ 
is $B=0$. 
Therefore we obtain the polynomial 
$f:=1+2x+x^2+xy^2 \in \mathcal{F}(\Delta_{\IV})$. 

\vspace{2mm}
\noindent
(V)\;
We can prove that the polynomial 
\[
f:=(1 \pm x)^4+y \in \mathcal{F}(\Delta_{\V})
\]
satisfies the condition.  
\end{proof}

Set
\begin{align*} 
&\Hat{\Delta}_{\III}:= \mathrm{Conv}\{ (0,-1), (2,0), (0,3) \}, \\
&\Hat{\Delta}_{\IV} := \mathrm{Conv}\{ (0,-2), (2,0), (0,2) \}, \\
&\Hat{\Delta}_{\V}  := \mathrm{Conv}\{ (0,-1), (4,0), (0,1) \}. 
\end{align*}

\begin{figure}[htbp]
\includegraphics[scale=0.6]{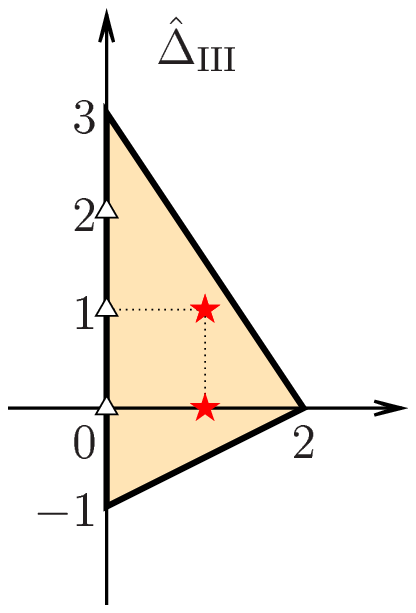}
\includegraphics[scale=0.6]{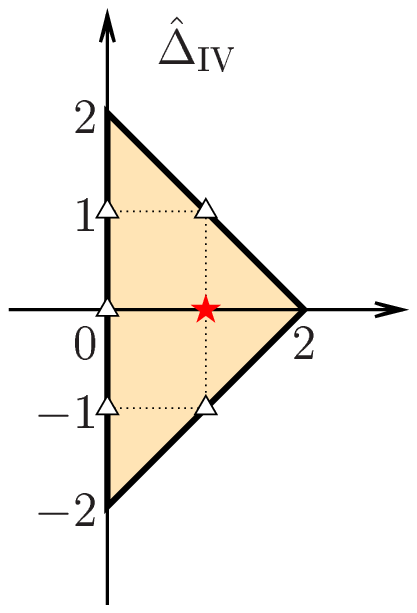}
\includegraphics[scale=0.6]{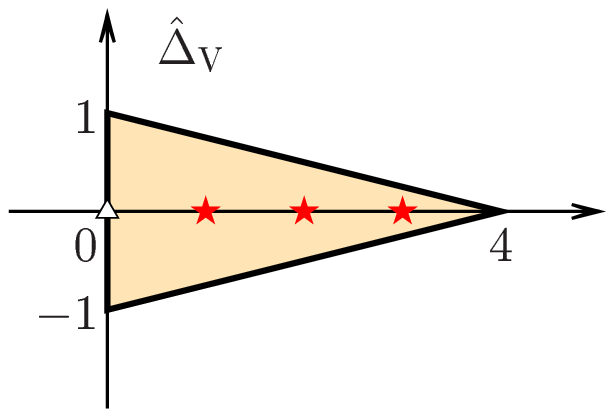}

\caption{Polytopes $\Hat{\Delta}_{\III}, \Hat{\Delta}_{\IV}$ and  
$\Hat{\Delta}_{\V}$. 
The notation {\tiny $\triangle$} means a lattice point on the boundary 
which is not a vertex 
and the notation \textcolor{red}{$\star$} means an interior lattice point.
} 
\end{figure}

For the polytopes $\Delta_{\III}$ and $\Delta_3(0;2,1,1)$ 
appearing in Definition~\ref{tropA3} (III), 
the polynomial on $\Delta_{\III}$ obtained in Lemma~\ref{singedge} 
induces the polynomial on $\Delta_3(0;2,1,1)$ as 
\[
1 +A x+x^2+y, 
\]
where $A=\pm2$.
Therefore the exceptional polytope in this case is 
$\hat{\Delta}_{\III}$.

For the polytopes $\Delta_{\IV}$ and $\Delta_3(1;2,1,1)$ 
appearing in Definition~\ref{tropA3} (IV), 
the polynomial on $\Delta_{\IV}$ obtained in Lemma~\ref{singedge} 
induces the polynomial on $\Delta_3(1;2,1,1)$ as
\[
1 +A x+x^2+Bxy+xy^2, 
\]
where $A=\pm2$. 
If $B = 0$, the exceptional polytope compatible with the data 
is $\hat{\Delta}_{\IV}$. 
Note that, if $B \neq 0$, the exceptional polytope compatible with the data 
is 
\[
\mathrm{Conv}\{ (2,0), (0,2), (0,-1) \}, 
\]
and it has no deformation pattern which defines an $1$-tacnodal curve, 
see the discussion in Lemma~\ref{nonirr}.

For the polytopes 
$\Delta_{\V}$ and $\Delta_3(0;4,1,1)$
appearing in Definition~\ref{tropA3} (V), 
the polynomial on $\Delta_{\V}$ obtained in Lemma~\ref{singedge} 
induces the same polynomial on $\Delta_3(0;4,1,1)$. 
Therefore, the exceptional polytope compatible with the data 
is $\hat{\Delta}_{\V}$.

\begin{lem}\label{refined}\it
For each $i=\III, \IV, \V$, 
there is a deformation pattern 
$\phi \in \mathcal{F}(\Hat{\Delta}_i)$ 
compatible with given data 
in Lemma \ref{singedge} 
which has the fixed coefficients on the vertices such that 
the curve defined by $\phi$ in $X(\hat{\Delta}_i)$ 
is a 1-tacnodal curve. 
\end{lem}
\begin{proof}
(III)\;
For the polynomial 
\[
\phi:=1+Ay+x^2y+By^2+Cxy^2+Dy^3+y^4 
\in \mathcal{F}({\hat{\Delta}_{\III}})
\]
we apply Lemma \ref{tacnode} and eliminate the variables by the system 
$\phi=\phi_x=\phi_y=\mathrm{Hess}(\phi)=K(\phi)=0$. 
Notice that $y$ is nonzero. 
First, by $\phi_x=0$, we can get $C$ as 
\[
C=-\frac{2x}{y}. 
\]
Therefore the system is reduced as 
\[
\begin{cases}
&\text{(1)}\;\; 1+Ay-x^2y+By^2+Dy^3+y^4    =0, \\
&\text{(2)}\;\; A-3x^2+2By+3Dy^2+4y^3      =0, \\
&\text{(3)}\:\; 4By-12x^2+12Dy^2+24y^3     =0,  \\
&\text{(4)}\:\; -x^2+Dy^2+4y^3             =0. 
\end{cases}
\]
Secondly, 
by equation (4), we obtain  
\[
x^2 = y^2(D+4y). 
\]
Then the system is reduced as
\[
\begin{cases}
&\text{(1')}\;\; 
1+Ay-3y^4+By^2   =0, \\
&\text{(2')}\;\; 
A-8y^3+2By       =0, \\
&\text{(3')}\;\;
-B+6y^2          =0. 
\end{cases}
\]
Thirdly, by equation (3'), we can get $B$ as 
\[
B=6y^2.  
\]
Then the system is reduced as 
\[
\begin{cases}
&\text{(1'')}\;\; 
1+Ay+3y^4  =0, \\
&\text{(2'')}\;\; 
A+4y^3     =0. 
\end{cases}
\]
Hence we obtain $A=-4y^3$ and then the equation 
\[ 
y^4-1=0. 
\]
The solution is 
\[
(A,B,C,D,x,y)
=
(-4y_0^3,6y_0^2,-2x_0/y_0,D,x_0,y_0),
\]
where 
$y_0$ is a solution of $y^4-1=0$ and 
$x_0$ is a solution of $x^2=y_0^2(D+4y_0)$. 

Next, we check that 
the above $\phi$ has only one singularity and it is a tacnode. 
Notice that the curve $V_{\phi}$ defined by $\phi$ 
has a tacnode at $(x_0, y_0) \in \C^2$. 
Let $(s,t) \in \C^2$ be a singular point of $V_{\phi}$. 
Then we solve 
$\phi(s,t)=\phi_x(s,t)=\phi_y(s,t)=0$ 
and we check that the solution is only $(s,t)=(x_0, y_0)$. 
That is, the singularity of $\phi$ is only one point and is a tacnode. 

\vspace{2mm}
\noindent
(IV)\;
We consider the following polynomial 
\[
\phi:= 1+Ay+By^2+Cy^3+y^4+c_{11}xy+c_{13}xy^3 +x^2y^2
\in \mathcal{F}(\hat{\Delta}_{\IV}).
\] 
Note that 
$c_{11}, c_{13}$ 
are both zero 
because of the form of the polynomials 
derived by Lemma \ref{singedge} (IV).

For the polynomial 
\[
\phi:= 1+Ay+By^2+Cy^3+y^4+x^2y^2
\in \mathcal{F}(\hat{\Delta}_{\IV}), 
\] 
we eliminate the variables by the system 
$\phi=\phi_x=\phi_y=\mathrm{Hess}(\phi)=K(\phi)=0$ 
by Lemma~\ref{tacnode}. 
Notice that $y$ is nonzero. 
First, by $\phi_x=0$, we obtain $x=0$. 
Therefore the system is reduced as 
\[
\begin{cases}
&\text{(1)}\;\; 1+Ay+By^2+Cy^3+y^4   =0, \\
&\text{(2)}\;\; A+2By+3Cy^2+4y^3     =0, \\
&\text{(3)}\:\; B+3Cy+6y^2           =0,  \\
&\text{(4)}\:\; C+4y                 =0. 
\end{cases}
\]
Secondly, by equation (4), we obtain  
\[
C=-4y.  
\]
Then the system is reduced as
\[
\begin{cases}
&\text{(1')}\;\; 
1+4y+By^2-3y^4  =0, \\
&\text{(2')}\;\; 
A+2By-8y^3       =0, \\
&\text{(3')}\;\;
B-6y^2          =0. 
\end{cases}
\]
Thirdly, by equation (3'), we can get $B$ as 
\[
B=6y^2.  
\]
Then the system is reduced as 
\[
\begin{cases}
&\text{(1'')}\;\; 
1+Ay+3y^4  =0, \\
&\text{(2'')}\;\; 
A+4y^3     =0. 
\end{cases}
\]
Hence we obtain $A=-4y^3$ and then the equation 
\[
y^4-1=0. 
\]
The solution is 
\[
(A,B,C,x,y)
=
(-4y_0^3,6y_0^2,-4y_0,0,y_0),
\]
where 
$y_0$ is a solution of $y^4-1=0$. 

Next, we check that 
the above $\phi$ has only one singularity and it is a tacnode. 
Notice that the curve $V_{\phi}$ defined by $\phi$ 
has a tacnode at $(0, y_0) \in \C^2$. 
Let $(s,t) \in \C^2$ be a singular point of $V_{\phi}$. 
Then, we solve 
$\phi(s,t)=\phi_x(s,t)=\phi_y(s,t)=0$ 
and check that the solution is only $(s,t)=(0, y_0)$. 
That is, the singularity of $\phi$ is only one point and is a tacnode. 

\vspace{2mm}
\noindent
(V)\;
In this case, 
in order to achieve $\phi_{xx} \neq 0$, 
we exchange the variables $x$ and $y$ in $\phi$. 

For the polynomial 
\[
\phi :=1+ Ax + Bxy + Cxy^2 + xy^4 + x^2
\in \mathcal{F}(\hat{\Delta}_{\V}),
\] 
we eliminate the variables by the system 
$\phi=\phi_x=\phi_y=\mathrm{Hess}(\phi)=K(\phi)=0$ 
by Lemma~\ref{tacnode}. 
Notice that $x$ is nonzero. 
First, by $\phi=0$, we obtain 
\[
A=-\frac{1+x^2+Bxy+Cxy^2+xy^4}{x}.
\]
Therefore the system is reduced as 
\[
\begin{cases}
&\text{(1)}\;\; (x-1)(x+1)   =0, \\
&\text{(2)}\;\; B+2Cy+4y^3   =0, \\
&\text{(3)}\:\; \text{substituting $A$ for }\mathrm{Hess}(\phi)  =0,  \\
&\text{(4)}\:\; \text{substituting $A$ for }K(\phi)  =0. 
\end{cases}
\]
Secondly, by equation (2), we obtain 
\[
B=-2y(C+2y). 
\]
Then the system is reduced as
\[
\begin{cases}
&\text{(1')}\;\; 
(x-1)(x+1)  =0, \\
&\text{(3')}\;\; 
4x(C+6y^2)       =0, \\
&\text{(4')}\;\;
192xy          =0. 
\end{cases}
\]
Thirdly, by equation (3'), we can get $C$ as 
\[
C=-6y^2.  
\]
The solution is 
\[
(A,B,C,x,y)
=
(\mp 2,0,0,\pm 1, 0). 
\]

Next, we check that 
the above $\phi$ has only one singularity and it is a tacnode. 
Suppose that the tacnode is at $(1,0)$. 
Let $(s,t) \in \C^2$ be a singular point of $V_{\phi}$. 
Then we solve $\phi(s,t)=\phi_x(s,t)=\phi_y(s,t)=0$ 
and check that the solution is only $(s,t)=(1,0)$. 
That is, the singularity of $\phi$ is only one point and is a tacnode. 
We can check the condition for the case where the tacnode is at 
$(-1,0)$ by the same manner. 
\end{proof}

\begin{rem}
Among the calculation in this section, 
there are finitely many polynomials which define $1$-tacnodal curves  
except case (III) in Lemma \ref{refined}. 
In case (III) in Lemma \ref{refined}, 
we can get the conclusion without eliminating the variable $D$. 
This means that there exists a one-parameter family of 
deformation patterns which define $1$-tacnodal curves. 
\end{rem}

\subsection{Remarks on the polytope $\Delta_{\E}$}

By the above discussion, for each tropical 1-tacnodal curve, except 
case (E), 
there is a ``\textit{degenerate model of 1-tacnodal curve}" 
whose tropical amoeba is the tropical 1-tacnodal curve. 
In this subsection, 
we discuss what happens in case (E).

\begin{lem}\label{expol_nontac} \it
There is NOT a polynomial $f \in \mathcal{F}(\Delta_{\E})$ which 
defines a 1-tacnodal curve on $X(\Delta_{\E})$. 
\end{lem}

\begin{proof}
We assume that a polynomial
\[
f:=c_{00}+Ax+c_{20}x^2+c_{01}y+Bxy+c_{12}xy^2
\]
defines a 1-tacnodal curve. 
Then, since $f_{xx}$ is non-zero, we can apply Lemma \ref{tacnode} and 
obtain 
$y=-B/2c_{12}$. 
Substituting it for 
$f_y=c_{01}+Bx+2c_{12}xy=0$, 
we get $c_{01}=0$, but this is a contradiction. 
\end{proof}

On the other hand, 
there is a polynomial 
$f \in \mathcal{F}(\Delta_{\E})$ 
which has a Newton degenerate singularity on 
$X(\sigma) \subset X(\Delta_{\E})$, 
where $\sigma \subset \Delta_{\E}$ is the edge of length 2. 
Actually, we can calculate as follows: 
Set 
$P:=\Delta_4(1;2,1,1,1), Q:=\Delta_3(0;2,1,1)$. 
We consider the polynomial 
\[
f:=c_{00}+Ax+c_{01}y+c_{20}x^2+Bxy+c_{12}xy^2 \in \mathcal{F}(P).
\]
By multiplying suitable constants to the variables 
and the whole polynomial, 
we can rewrite $f$ as 
\[
1+Ax+y+x^2+Bxy+Cxy^2 \in \mathcal{F}(P).
\]
If the curve $V_f\subset X(P)$ defined by $f$ intersects $X(\sigma)$ 
at two points, 
we can easily check that these points are smooth points of $V_f$ 
and the intersection $V_f \cap X(\sigma)$ is transversal. 
Therefore $V_f \cap X(\sigma)$ is exactly one point. 
Then, $f$ can be rewritten as follows:  
\[
f=(\epsilon+x)^2+y+Bxy+Cxy^2 \in \mathcal{F}(P), 
\]
where $\epsilon = \pm 1$. 
Set $(X,Y):=(x + \epsilon ,y)$. 
Then $f$ is rewritten as follows: 
\[
\tilde{f}(X,Y):=
X^2+BXY+(1 \mp B)Y+CXY^2 \mp CY^2. 
\]

Thus the most complicated isolated singular point 
defined by this polynomial 
at the origin 
(under the condition that the form of the polynomial does not change) 
is given as 
\[
X^2 \pm XY + \frac{1}{4}Y^2 +\text{(higher terms)}.  
\] 
More precisely, since the polynomial $f$ is irreducible, 
the number of interior lattice points of $\Delta_{\E}$ is two and 
the curve defined by $f$ has no singularity more complicated than $A_3$, 
the curve has only a cusp as the singularity.

Applying mechanically refinement arguments in this case, 
we find that the edge 
$\Delta_4(1;2,1,1,1) \cap \Delta_3(0;2,1,1)$ 
does not correspond a 1-tacnodal curve as follows:
By above discussion, the exceptional polytope in this case is 
$\hat{\Delta}_2$. 
We only consider the case of 
$\epsilon=1$. 
The other case can be proved by the same argument. 
According to the explanation of a deformation pattern in 
Definition~\ref{defpat},
we set 
\[
\phi:=1+A'y+x^2y+B'y^2 +xy^2+\frac{1}{4}y^3 
\in \mathcal{F}(\hat{\Delta}_2). 
\]
By direct computation, we get $\phi_{xx} \neq 0$. 
Using Lemma~\ref{tacnode}, 
we obtain $8B'x = 0$. 
Both cases $x=0$ and $B'=0$ contradict $\phi=0$.

In \cite{S}, it is assumed that 
each polynomial $f_i$ has only semi-quasi-homogeneous singularity 
since the paper only deals with the case of nodal or 1-cuspidal curves. 
This assumption may not be reasonable in the case of 1-tacnodal curves. 
Actually, when we list the possible polytopes for 1-tacnodal curves 
we cannot ignore case (E). 
This is the reason why this case is included in the definition 
of tropical 1-tacnodal curves. 
Note that, in fact, by the above discussion, 
there is no degenerate model of 1-tacnodal curve corresponding to case (E).

\section{Main Result}

The main theorem of this paper is the following:

\begin{thm}\label{thm1}\it
Let $F \in K[z,w]$ be a polynomial 
which defines an irreducible 1-tacnodal curve. 
If the rank of the tropical amoeba $T_F$ defined by $F$ 
is more than or equal to 
the number of the lattice points of the Newton polytope of 
$F$ minus four, 
then $T_F$ is a tropical 1-tacnodal curve. 
\end{thm}

Let $F$ be a polynomial in the assertion, 
$T_F$ be the tropical amoeba defined by $F$, 
whose rank satisfies 
\[
\sharp \Delta_{\Z} -1 
\ge 
\rk(T_F)
\ge 
\sharp \Delta_{\Z}-4,
\] 
and $S$ be the dual subdivision of $T_F$. 
We remark that, from the discussion in \cite[Section 4]{S}, 
if 
$\sharp \Delta_{\Z} -1 \ge 
\rk(T_F)
\ge \sharp \Delta_{\Z}-3$, 
then $T_F$ is smooth, nodal or 1-cuspidal. 
Thus, by Remark~\ref{KR}, 
we can assume that the rank of $T_F$ is $\sharp \Delta_{\Z}-4$. 

From the discussion in \cite[Subsection 3.3]{S} 
and the equality $g(C^{(t)}) = \sharp \mathrm{Int} \Delta_{\Z}-2$, 
we can see that
\[
   \sharp \partial\Delta_{\Z} - \sharp ( V(S) \cap \partial\Delta)
   =0 \; \text{or} \; 1.
\]
We decompose the proof into four cases
\begin{itemize}
\item[(A)]
$S$ is a TP-subdivision and satisfies 
$\sharp \partial \Delta_{\Z} - \sharp(V(S)\cap \partial \Delta)=0$, 
\item[(B)]
$S$ is a TP-subdivision and satisfies 
$\sharp \partial \Delta_{\Z} - \sharp(V(S)\cap \partial \Delta)=1$,
\item[(C)]
$S$ is NOT a TP-subdivision and satisfies 
$\sharp \partial \Delta_{\Z} - \sharp(V(S)\cap \partial \Delta)=0$,
\item[(D)]
$S$ is NOT a TP-subdivision and satisfies 
$\sharp \partial \Delta_{\Z} - \sharp(V(S)\cap \partial \Delta)=1$. 
\end{itemize}
For each case, we remove polytopes which cannot correspond 
to a $1$-tacnodal curve and show that the remaining polytopes 
are exactly tropical $1$-tacnodal curves in Definition~\ref{tropA3}.

To explain the removing process more precisely, 
we prepare some terminologies. 

\begin{defi}
A 2-dimensional polytope $P$ is \textit{1-tacnodal} if 
there is a polynomial 
$f \in \mathcal{F}(P)$ 
which defines a 1-tacnodal curve 
$V_f \in |D(P)|$ 
satisfying the conditions (S1) and (S2) in Subsection \ref{existtac}. 

Let $\sigma:=P_1 \cap P_2$ be an edge which is 
the intersection of 2-dimensional polytopes $P_1$ and $P_2$. 
The edge $\sigma$ is \textit{1-tacnodal} 
if there is a pair of polynomials 
$(f_1, f_2) \in \mathcal{F}(P_1) \times \mathcal{F}(P_2)$ 
such that 
\begin{itemize}
\item  
their truncation polynomials 
$f_1^{\sigma}, f_2^{\sigma}$ 
on $\sigma$ are same, 

\item  
each of the curves $C_1$ and $C_2$ 
defined by $f_1$ and $f_2$ 
has a smooth point or an isolated singular point at $z$ in $X(\sigma)$, 

\item  
there exists a deformation pattern 
$\phi \in \mathcal{F}(\Delta_{z})$ 
compatible with the above data 
which defines a 1-tacnodal curve in $X(\Delta_{z})$. 
\end{itemize}
\end{defi}

It can be seen from the discussion in Subsection~\ref{existtac} that 
the polytopes and the pairs of polytopes 
appearing in Definition \ref{tropA3} are $1$-tacnodal. 
To prove the theorem, for each of cases (A), (B), (C) and (D),
we carry out the following arguments.
\begin{itemize}
\item[(1)] 
Remove configurations of edges 
and interior lattice points of polytopes which do not exist.
\item[(2)] 
Classify polytopes that are not $1$-tacnodal.
\item[(3)] 
From the list in (2), remove polytopes 
which do not have $1$-tacnodal edges.
\end{itemize}

In Subsection 4.1, 
we prepare lemmata for the non-existence of polytopes in (1), 
and then prove the theorem for case (A), (B), (C) and (D) 
in Subsection 4.2, 4.3, 4.4 and 4.5, respectively.

\subsection{Auxiliary definitions and lemmata}

\begin{lem}[On interior lattice points] \label{poly2} \it
(1)\; The number of interior lattice points of non-parallel quadrangle 
whose edges are length $1$ is larger than $0$. \\
(2)\; For an integer $m \ge 5$, 
the number of interior lattice points of an $m$-gon is larger than $0$.
\end{lem}

\begin{proof}
(1)\;
If a non-parallel $\Delta_4(0;1,1,1,1)$ exists, 
it can be decomposed into two triangles of area $1/2$.
Thus, we can map this polytope to 
\[\mathrm{Conv}\{ (0,0),(1,0),(0,1),(p,q) \}\]
by some isomorphism.
Then, from Pick's formula, we obtain 
\[ \frac{p+q}{2}=1. \]
Hence $p=q=1$. This is a parallelogram. \\
(2)\;
It is obvious from the facts that 
the minimum pentagon is $\Delta_ {\VII}$ 
and 
any $m$-gon can be decomposed into polytopes including a pentagon. 
\end{proof}

\begin{lem}[Non-existence of some polytopes] \label{polyt1}\it
(1)\; Following polytopes do NOT exist:
\[ 
\Delta_3(1;2,2,1), \;\; 
\Delta_3(1;3,1,1), \;\;
\Delta_3(0;2,2,1), \;\;
\Delta_3(0;3,2,1), \;\;
\Delta_5(0;2,1,1,1,1). 
\]
\noindent
(2) \; 
There is NO non-parallel quadrangle 
$\Delta_4(0;2,2,1,1)$. 
\end{lem}
\begin{proof}
(1)\;
The first triangle is equivalent to 
\[ \mathrm{Conv}\{ (p,0),(p+2,0),(0,q) \}.  \]
By Pick's formula, 
we obtain $q=5/2$. 
But this contradicts $q \in \Z$. 
We can easily check the non-existence of the 
second, third and fourth triangles.
If there exists a pentagon 
$\Delta_5(0;2,1,1,1,1)$, 
we can split it into two quadrangles 
$\Delta_4(0;1,1,1,1)$ and $\Delta'_4(0;1,1,1,1)$. 
But these quadrangles are parallelograms by the fact that 
already proved in Lemma~\ref{poly2} (1). 
Thus the union can not be a pentagon. 

\vspace{2mm}
\noindent
(2) \; 
If it exists, then the edges of length $2$ are either 
adjacent or in opposite sides. 
The former case can not occur 
since a triangle $\Delta_3(0;2,2,1)$ does not exist.  
In the latter case, 
we can split it into two triangles 
$\Delta_{3}(0;2,1,1)$ and $\Delta'_{3}(0;2,1,1)$. 
We can assume that one of the triangles 
is isomorphic to 
$\mathrm{Conv} \{(0,0), (1,0), (0,2) \}$ 
and the common edge is the bottom edge.
Then, by Pick's formula, 
the last vertex of 
$\Delta_4(0;2,2,1,1)$ 
must be one of the following lattice points 
\[ (0,-2),\;\; (1,-2),\;\; (2,-2), \] 
but all of them do not satisfy the required conditions.
\end{proof}

\begin{lem}\it \label{nonisol}
For the polytope 
\[
P:=\mathrm{Conv}\{(0,0), (2,0), (0,1), (2,1)\}, 
\]
the polynomial 
\[
f :=c_{00}+Ax+c_{20}x^2+c_{01}y+Bxy+c_{21}x^2y
\in \mathcal{F}(P)
\]
satisfies $f=f_x=f_y=\mathrm{Hess}(f)=0$ 
if and only if 
\[ c_{21}c_{00}=c_{20}c_{01}.  \]
Moreover, if $f$ satisfies 
$f=f_x=f_y=\mathrm{Hess}(f)=0$, 
i.e., $V_f \subset X(\Delta_X)$ 
has a singularity more complicated than $A_2$,
then $f$ has the form 
\[
   (y+1)(x\pm 1)^2
\]
up to multiplication of a non-zero constant.
In particular, the set of singularities of $V_f$ is non-isolated.
\end{lem}

\begin{proof}
By direct computation. 
\end{proof}

\begin{rem}[Known Results]\label{KR}
(1) \; Let 
$I \ge 0$, $s,t,u \ge 1$ 
be integers such that 
\[
0 \le I+(s-1)+(t-1)+(u-1) \le 2. 
\]
For each $(I;s,t,u)$, 
a triangle 
$\Delta_3(I;s,t,u)$ 
is uniquely determined up to 
$\mathrm{Aff}(\Z^2)$-equivalence 
as follows: 
\begin{align*}
\Delta_3(2;1,1,1)
&\simeq 
\mathrm{Conv}\{ (0,0), (3,2), (2,3) \}, \\
\Delta_3(1;2,1,1)
&\simeq 
\mathrm{Conv}\{ (0,0), (2,0), (1,2) \}, \\
\Delta_3(1;1,1,1)
&\simeq 
\mathrm{Conv}\{ (0,0), (1,2), (2,1) \}, \\
\Delta_3(0;3,1,1)
&\simeq 
\mathrm{Conv}\{ (0,0), (3,0), (0,1) \}, \\
\Delta_3(0;2,1,1)
&\simeq 
\mathrm{Conv}\{ (0,0), (2,0), (0,1) \}, \\
\Delta_3(0;1,1,1)
&\simeq 
\mathrm{Conv}\{ (0,0), (1,0), (0,1) \}.
\end{align*}
\noindent
(2)\;
\; 
For integers 
$I \in \{ 0,1 \}$, $s,t \ge 1$ such that 
\[
0 \le I+2(s-1)+2(t-1) \le 2, 
\]
a parallelogram 
$\Delta^{\mathrm{par}}_4(I;s,t)$
is uniquely determined up to 
$\mathrm{Aff}(\Z^2)$-equivalence 
as follows: 
\begin{align*}
\Delta^{\mathrm{par}}_4(1;1,1)
&\simeq 
\mathrm{Conv}\{ (0,0), (1,0), (1,2), (2,2) \}, \\
\Delta^{\mathrm{par}}_4(0;2,1)
&\simeq 
\mathrm{Conv}\{ (0,0), (2,0), (0,1), (2,1) \}, \\
\Delta^{\mathrm{par}}_4(0;1,1)
&\simeq 
\mathrm{Conv}\{ (0,0), (1,0), (0,1), (1,1) \}. 
\end{align*}
\noindent
(3)\;
The polytopes in this remark are not 1-tacnodal 
(By \cite[Lemma 4.2]{S} and Lemma \ref{nonisol}, or direct computation). 
\end{rem}

\begin{lem}[Describing some polytopes] \label{desc} \it
(1)\;
Let 
$I \ge 0$, $s,t,u \ge 1$ 
be integers such that 
\[
I+(s-1)+(t-1)+(u-1) =3.
\]
For each $(I;s,t,u)$, 
a triangle 
$\Delta_3(I;s,t,u)$
has the following isomorphisms: 
\begin{align*}
\Delta_3(3;1,1,1)
&\simeq 
\Delta_{\I}, \Delta_{\II}, \\
\Delta_3(2;2,1,1)
&\simeq 
\mathrm{Conv}\{ (0,0), (2,0), (1,3) \}, \\
\Delta_3(0;4,1,1)
&\simeq 
\mathrm{Conv}\{ (0,0), (0,1), (4,0) \}, \\
\Delta_3(0;2,2,2)
&\simeq 
\mathrm{Conv}\{ (0,0), (2,0), (0,2) \}. 
\end{align*}
\noindent
(2)\;
A quadrangle 
$\Delta_4(0;2,1,1,1)$
is uniquely determined as 
$ \mathrm{Conv}\{ (0,0), (2,0), (0,1), (1,1) \}$ 
up to $\mathrm{Aff}(\Z^2)$-equivalence. 
\end{lem}

\begin{proof}
(1)\;
These claims, except the last case, 
are the same as Lemma \ref{unique}.
We prove the last one. 
Without loss of generality, 
the polytope can be assumed to be 
\[ 
\mathrm{Conv}\{ (p,0), (p+2,0) ,(0,q) \}. 
\]
From Pick's formula, we obtain $q=2$ and 
$p=2k$ for some $k \in \Z$. 
Thus, by the isomorphism
\[
\begin{pmatrix}
1 & k \\
0 & 1 \\
\end{pmatrix}
: \Z^2 \to \Z^2,
\]
it is mapped to the polytope 
$\mathrm{Conv}\{ (0,0), (2,0), (0,2) \}$. 

\vspace{2mm}
\noindent
(2)\;
We can split 
$P=\Delta_4(0;2,1,1,1)$
into two polytopes $Q$, $R$ which are either \\
\textbullet \; 
$Q=\Delta_3(0;2,1,1)$, $R=\Delta_3(0;1,1,1)$
and these polytopes share an edge of length $1$, or\\
\textbullet \; 
$Q=\Delta_3(0;1,1,1)$, $R=\Delta_4(0;1,1,1,1)$
and these polytopes share an edge of length $1$. 

In the former case, we can assume that 
$Q$ is 
\[
\mathrm{Conv}\{ (0,0), (1,0), (0,2) \}
\]
and the common edge is its bottom edge. 
Then the last vertex of $P$ must be 
$(1,-1)$. 
In the latter case, 
we can assume that 
$R$ is 
\[
\mathrm{Conv}\{ (0,0), (1,0), (0,1), (1,1) \}
\]
and the common edge is its bottom edge. 
Then the last vertex of $P$ must be either 
$(0,-1)$, or $(1,-1)$. 
All of them are equivalent to 
\[
\mathrm{Conv}\{ (0,0), (2,0), (0,1), (1,1) \}. 
\]
\end{proof}

\begin{lem}[Non 1-tacnodal polytopes]\label{nonreg}\it
The following polytopes are not 1-tacnodal polytopes: \\
(1)\;
$\Delta_3(0;2,2,2)$, \\
(2)\;
$\Delta_3(0;4,1,1)$, \\
(3)\;
$\Delta_3(2;2,1,1)$, \\
(4)\;
$\Delta_4(0;2,1,1,1)$, \\
(5)\;
$\mathrm{Conv}\{ (1,0), (0,1), (2,1), (1,3) \}$. 
\end{lem}

\begin{proof}
(1)\;
This is by the fact that 
the Milnor number of an isolated singularity 
of a projective conic does not exceed $1$. 

\vspace{2mm}
\noindent
(2)\;
Notice that this polytope is uniquely determined as 
$\mathrm{Conv}\{ (0,0), (0,1), (4,0) \}$. 
Then a polynomial $f$ with this Newton polytope has no singularity
since $f_y$ is a non-zero constant. 

\vspace{2mm}
\noindent
(3)\;
We assume that a polynomial 
\[ 
f:=1+Ax+x^2+Bxy+Cxy^2+xy^3 \in \mathcal{F}(\Delta_{3}(2;2,1,1))
\]
satisfies the condition (S1). 
Since the polynomial $f$ satisfies $f_{xx} \neq 0$, 
the system $f=f_x=f_y=\mathrm{Hess}(f)=K(f)=0$ must have a solution. 
But, we obtain $K(f)=48x$. This is a contradiction. 

\vspace{2mm}
\noindent
(4)\;
Notice that this polytope is uniquely determined as 
$\mathrm{Conv}\{ (0,0), (2,0), (0,1), (1,1) \}$. 
We set a polynomial $f$ as 
\[
f:=c_{00}+Ax+c_{20}x^2+c_{01}y+c_{11}xy \in 
\mathcal{F}(\mathrm{Conv}\{ (0,0), (2,0), (0,1), (1,1) \}). 
\]
The hessian of $f$ is $-c_{11}^2 \neq 0$.

\vspace{2mm}
\noindent
(5)\; 
We assume that a polynomial
\[
f:=c_{10}x+c_{01}y+Axy+c_{21}x^2y+Bxy^2+c_{13}xy^3
\]
defines a 1-tacnodal curve. 
Then, since $f_{xx}$ is non-zero, we can apply~Lemma \ref{tacnode} and 
obtain 
\[
4c_{01}x(c_{13}y^3+c_{10})=-4c_{01}c_{13}xy^3=0. 
\] 
This is a contradiction. 
\end{proof}

Set 
\begin{align*}
\hat{\Delta}_1
&=\mathrm{Conv}\{ (2,0),(0,1),(0,-1) \}, \\
\hat{\Delta}_2
&=\mathrm{Conv}\{ (2,0),(0,2),(0,-1) \}, \\
\hat{\Delta}_3
&=\mathrm{Conv}\{ (3,0),(0,1),(0,-1) \}, 
\end{align*}
see Figure~4.

\begin{figure}[htbp]\label{hats}
\includegraphics[scale=0.8]{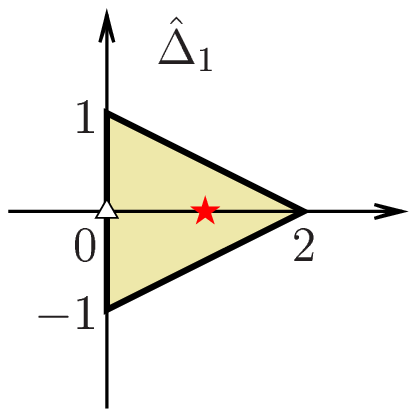}
\includegraphics[scale=0.8]{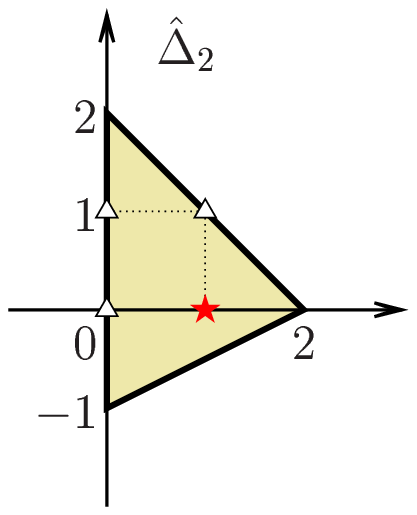}
\includegraphics[scale=0.8]{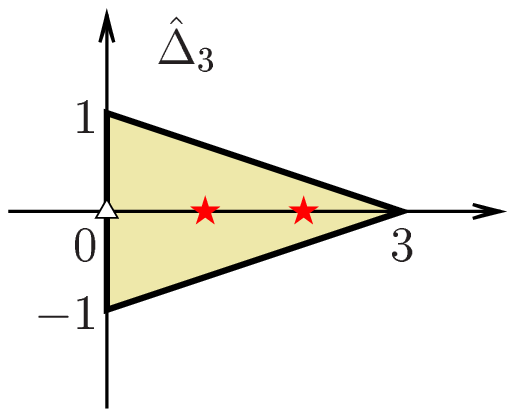}
\caption{
Polytopes $\Hat{\Delta}_{1}, \Hat{\Delta}_{2}$ and $\Hat{\Delta}_{3}$. 
The notation {\tiny $\triangle$} means a lattice point on the boundary 
which is not a vertex 
and the notation \textcolor{red}{$\star$} means an interior lattice point.
}
\end{figure}

\begin{lem}[Non 1-tacnodal edges]\label{nonirr}\it
The following edges $\sigma$ are not 1-tacnodal edges:\\
(1)\;
the edge 
$\Delta_3(0;2,1,1) \cap \Delta_3(0;2,1,1)$
of length $2$, \\
(2)\;
the edge 
$\Delta_3(1;2,1,1) \cap \Delta_3(0;2,1,1)$ 
of length $2$ 
and the edge 
$\Delta_3(1;2,1,1) \cap \Delta_4(0;2,1,1,1)$
of length $2$, \\
(3)\;
the edge 
$\Delta_3(0;3,1,1) \cap \Delta_3(0;3,1,1)$
of length $3$, \\
(4)\;
the edge 
$\Delta_4(0;2,1,1,1) \cap \Delta_3(0;2,1,1)$
of length $2$, \\
(5)\;
the edge 
$\Delta_4(0;2,1,1,1) \cap \Delta_4(0;2,1,1,1)$ 
of length $2$, \\
(6)\;
the edge 
$\Delta^{\mathrm{par}}_4(0;2,1) \cap \Delta_3(0;2,1,1)$
of length $2$ 
and the edge 
$\Delta^{\mathrm{par}}_4(0;2,1) \cap \Delta_4(0;2,1,1,1)$ 
of length $2$, \\
(7)\;
the edge 
$\Delta_3(0;2,2,2) \cap \Delta_3(0;2,1,1)$ 
of length $2$. 
\end{lem}

\begin{proof}
The assertion for cases 
(1), (3), (4) and (5) 
are already proved in \cite[Lemma 3.9, 3.10 and 4.4]{S}.
Here we only prove (2), (6) and (7). \\
(2)\; 
Set 
$P:=\Delta_3(1;2,1,1), Q:=\Delta_3(0;2,1,1)$. 
It is easy to check that 
a curve in $|D(P)|$ cannot have an isolated singularity
more complicated than $A_1$. 
Also, we can easily check that 
if a curve $V_f$ intersects $X(\sigma)$ at two points 
then the points are smooth points of $V_f$ and 
those intersections are transversal.
Therefore we can set 
\[ 
f:=(x+\epsilon)^2+Axy+xy^2 \in \mathcal{F}(P),  
\]
where $\epsilon = \pm 1$ 
and suppose that $f$ defines 
a curve which has an $A_1$-singularity on $X(\sigma) \subset X(P)$. 
With a simple calculation, we obtain $A=0$. 
The polynomial corresponding to the polytope $Q$ becomes 
\[
f':=(x+\epsilon)^2+y \in \mathcal{F}(Q). 
\] 
Then the exceptional polytope in this case is 
$\hat{\Delta}_2$. 
According to the explanation of a deformation pattern in 
Definition~\ref{defpat},
we set 
\[
\phi:=1+A'y+\epsilon x^2y+B'y^2+y^3 \in \mathcal{F}(\hat{\Delta}_2). 
\]
In the case $\epsilon=1$, we get $\phi_{xx}\ne 0$ by $y\ne 0$.
Using Lemma \ref{tacnode}, 
we obtain $48y^3=0$,  
but this is a contradiction. 
We also have a contradiction in the case $\epsilon=-1$.

\vspace{2mm}
\noindent
(6)\;
Set 
$P:=\Delta^{\mathrm{par}}_4(0;2,1), Q:=\Delta_3(0;2,1,1)$. 
For $P$, 
we set 
\[
f:=(\epsilon+x)^2 + (1+Ax+x^2)y
\in \mathcal{F}(P),  
\]
where $\epsilon= \pm 1$. 
Then a polynomial corresponding to $Q$ 
must be 
\[
f':=(\epsilon+x)^2+y
\in \mathcal{F}(Q).
\]
Then the exceptional polytope in this case is 
$\hat{\Delta}_1$. 

If $\epsilon=1$, $\phi$ is given as  
\[
\phi:=1+x^2y+y^2+A'y
\in \mathcal{F}(\hat{\Delta}_1), 
\]
and we can easily check that the solution of the system 
$\phi=\phi_x=\phi_y=\mathrm{Hess}(\phi)=0$ 
does not exist. 
The case $\epsilon=-1$ can be proved by the same argument.

\vspace{2mm}
\noindent
(7)\;
Set 
$P:=\Delta_3(0;2,2,2), Q:=\Delta_3(0;2,1,1)$. 
Without loss of generality, 
we can assume that $P$ and $Q$ are 
\[
P=\mathrm{Conv}\{(0,0), (2,0), (0,2)\}, 
\quad 
Q=\mathrm{Conv}\{(0,0), (2,0), (0,-1)\}. 
\] 
For $P$, 
we set 
\[
f:=
1+ 2 \epsilon x +x^2 +By+y^2+Cxy 
\in \mathcal{F}(P),  
\]
where $\epsilon= \pm 1$. 
Applying the new coordinates $(X,Y):=(x+\epsilon, y)$ for $f$, 
we obtain 
\[
f=X^2+(B-C\epsilon)Y+CXY+Y^2. 
\]
Notice that, if $\mathrm{Hess}(f)=C^2-4 =0$, 
$f$ defines a line of multiplicity $2$, 
that is, $f$ has non-isolated singularity. 
Therefore we may assume $C^2-4 \neq 0$. 
If $B-C\epsilon \neq 0$, 
the exceptional polytope in this case is $\hat{\Delta}_1$. 
If $B-C\epsilon =0$, then $(\epsilon,0)\in \C^2$ is an $A_1$-singularity, 
i.e., $f$ has the form $f=X^2+CXY+Y^2$. 
Hence, the exceptional polytope in this case is 
$\hat{\Delta}_2$. 
The conclusion is derived by the same calculation as 
in (7) for the former case and 
in (2) for the latter case, respectively.
\end{proof}

\begin{rem}[On an edge of length $1$] \label{onelength}
Let $\Delta_1, \Delta_2$ be polytopes such that 
their intersection $\sigma:=\Delta_1 \cap \Delta_2$ 
is an edge of length $1$. 
The edge $\sigma$ is NOT an $1$-tacnodal edge. 
Actually, we can prove it as follows: 
For integers $m_1, m_2 >0$ and 
the triangle 
\[
\hat{\Delta}:=
\mathrm{Conv}\{(1,0), (0,m_1), (0,-m_2)\}, 
\]
a polynomial 
$\phi \in \mathcal{F}(\hat{\Delta})$ 
can be given as 
\[
\phi=1+\psi(y)+xy^{m_2}, 
\]
where $\psi \in \C[y]$ is a polynomial in $y$ 
which satisfies $\psi(0)=0$. 
If the polynomial $\phi$ defines a singular curve, 
then $\phi=\phi_x=\phi_y=0$ at the singular point. 
By $\phi_{x}=y^{m_2}=0$, 
the singular point satisfies $y=0$. 
However it satisfies 
$\phi(x,0) \neq 0$ and this is a contradiction. 
Therefore, any deformation pattern cannot define 
a $1$-tacnodal curve. 
\end{rem}

To prevent complication of the proof of the main theorem, 
we give the following auxiliary definition. 

\begin{defi}
The notation $\mathbb{T}_{-1}$ means the set 
of polytopes equivalent to $\Delta_3(1;1,1,1)$ 
and pairs of polytopes equivalent 
to the pair of 
$\Delta_3(0;2,1,1)$ and $\Delta'_3(0;2,1,1)$
such that their intersection 
$\Delta_3(0;2,1,1) \cap \Delta'_3(0;2,1,1)$ 
is a segment of length $2$. 

The notation $\mathbb{T}_{-2}$ means the set of polytopes 
equivalent to
$\Delta_3(2;1,1,1)$ 
and pairs of polytopes equivalent to either 
\begin{itemize}
\item
the pair of $\Delta_3(1;2,1,1)$ and $\Delta_3(0;2,1,1)$ 
such that their intersection 
$\Delta_3(1;2,1,1) \cap \Delta_3(0;2,1,1)$ 
is a segment of length $2$, 

\item 
the pair of $\Delta_3(0;3,1,1)$ and $\Delta'_3(0;3,1,1)$ 
such that their intersection 
$\Delta_3(0;3,1,1) \cap \Delta'_3(0;3,1,1)$ 
is a segment of length $3$. 

\end{itemize}
\end{defi}

The triple 
$\Delta_3(0;2,2,1)$, $\Delta_3(0;2,1,1)$ and $\Delta'_3(0;2,1,1)$ 
such that the intersections 
$\Delta_3(0;2,2,1) \cap \Delta_3(0;2,1,1)$ 
and 
$\Delta_3(0;2,2,1) \cap \Delta'_3(0;2,1,1)$ 
are segments of length $2$ does not exist by Lemma~\ref{polyt1}.

Note that, from the above discussion, 
these polytopes and their sharing edges are not 1-tacnodal.

\subsection{Case (A)}
Let $S$ be the dual subdivision of $T_F$. 
We assume that 
$S$ is a TP-subdivision and satisfies 
$\sharp \partial \Delta_{\Z} - \sharp(V(S)\cap \partial \Delta)=0$. 
Then $d(S)=0$ by Lemma \ref{rklem}. 
Thus 
\[ \rk(T_F)=\rkexp(T_F)=\sharp \Delta_{\Z}-4. \] 
By the definition of $\rkexp(T_F)$, we get
\begin{align*}
\sharp \Delta_{\Z}-4
&=\sharp V(S)-1-\sum_{k=1}^N(\sharp V(\Delta_k)-3)\\
&=\sharp V(S)-1-N_4'.
\end{align*}
Since $\sharp V(S) \le \sharp \Delta_{\Z}$, 
we obtain $0 \le N_4' \le 3$. 

\begin{itemize}
\item[(A-0)]
If $S$ satisfies $N_4'=0$, 
then it satisfies $\sharp V(S)=\sharp \Delta_{\Z} -3$ and 
consists of triangles. 
Then, the subdivision $S$ must contain exactly one of 
the following polytopes:
\begin{itemize}
\item[(i)] 
$\Delta_3(3;1,1,1)$,
\item[(ii)] 
$\Delta_3(2;1,1,1)$ with one of $\mathbb{T}_{-1}$, 
\item[(iii)]
$\Delta_3(2;2,1,1)$ and $\Delta_3(0;2,1,1)$ 
such that their intersection is a segment whose length is $2$,
\item[(iv)]
$\Delta_3(1;2,1,1)$ and $\Delta_3(1;2,1,1)$ 
such that their intersection is a segment whose length is $2$,
\item[(v)]
$\Delta_3(1;2,1,1)$ and $\Delta_3(0;2,1,1)$ 
such that their intersection is a segment whose length is $2$
with one of $\mathbb{T}_{-1}$, 
\item[(vi)]
$\Delta_3(1;2,2,1)$, $\Delta_3(0;2,1,1)$ and $\Delta'_3(0;2,1,1)$ 
such that their intersections 
$\Delta_3(1;2,2,1) \cap \Delta_3(0;2,1,1)$ and 
$\Delta_3(1;2,2,1) \cap \Delta'_3(0;2,1,1)$
are segments whose lengths are $2$,
\item[(vii)]
$\Delta_3(0;2,2,1)$, $\Delta_3(0;2,1,1)$ and $\Delta_3(1;2,1,1)$ 
such that their intersections 
$\Delta_3(0;2,2,1) \cap \Delta_3(0;2,1,1)$ and 
$\Delta_3(0;2,2,1) \cap \Delta_3(1;2,1,1)$ 
are segments whose lengths are $2$, 
\item[(viii)]
$\Delta_3(0;2,2,2)$, $\Delta_3(0;2,1,1)$, $\Delta'_3(0;2,1,1)$ 
and $\Delta''_3(0;2,1,1)$ 
such that their intersections 
$\Delta_3(0;2,2,2) \cap \Delta_3(0;2,1,1)$, 
$\Delta_3(0;2,2,2) \cap \Delta'_3(0;2,1,1)$ and 
$\Delta_3(0;2,2,2) \cap \Delta''_3(0;2,1,1)$ 
are segments whose lengths are $2$, 
\item[(ix)]
$\Delta_3(1;3,1,1)$ and $\Delta_3(0;3,1,1)$ 
such that their intersection 
$\Delta_3(1;3,1,1) \cap \Delta_3(0;3,1,1)$ 
is a segment whose length is $3$, 
\item[(x)]
$\Delta_3(0;3,1,1)$ and $\Delta'_3(0;3,1,1)$ 
such that their intersection 
$\Delta_3(0;3,1,1) \cap \Delta_3(0;3,1,1)$ 
is a segment whose length is $3$, 
with one of $\mathbb{T}_{-1}$, 
\item[(xi)]
$\Delta_3(0;3,2,1)$, $\Delta_3(0;3,1,1)$ and $\Delta_3(0;2,1,1)$ 
such that their intersections 
$\Delta_3(0;3,2,1) \cap \Delta_3(0;3,1,1)$ and 
$\Delta_3(0;3,2,1) \cap \Delta_3(0;2,1,1)$ 
are segments whose lengths are $3$ and $2$, respectively, 
\item[(xii)]
$\Delta_3(0;4,1,1)$ and $\Delta'_3(0;4,1,1)$ 
such that their intersection 
$\Delta_3(0;4,1,1) \cap \Delta'_3(0;4,1,1)$ 
is a segment whose length is $4$, 
\item[(xiii)]
three of $\mathbb{T}_{-1}$, 
\item[(xiv)]
one of $\mathbb{T}_{-2}$ and 
one of $\mathbb{T}_{-1}$. 
\end{itemize}

\item[(A-1)]
If $S$ satisfies $N_4'=1$, 
then it satisfies $\sharp V(S)=\sharp \Delta_{\Z} -2$ and 
contains only one parallelogram in the following list and 
the rest of $S$ consists of triangles: 
\begin{itemize}
\item[(i)]
$\Delta^{\mathrm{par}}_4(2;1,1)$, 
\item[(ii)]
$\Delta^{\mathrm{par}}_4(0;2,1)$, 
$\Delta_3(0;2,1,1)$ and $\Delta'_3(0;2,1,1)$
such that their intersections 
$\Delta^{\mathrm{par}}_4(0;2,1) \cap \Delta_3(0;2,1,1)$ and 
$\Delta^{\mathrm{par}}_4(0;2,1) \cap \Delta'_3(0;2,1,1)$
are segments whose lengths are $2$, 
\item[(iii)]
$\Delta^{\mathrm{par}}_4(1;1,1)$ 
with one of $\mathbb{T}_{-1}$, 
\item[(iv)]
$\Delta^{\mathrm{par}}_4(0;1,1)$ 
with two of $\mathbb{T}_{-1}$, 
\item[(v)]
$\Delta^{\mathrm{par}}_4(0;1,1)$ 
with one of $\mathbb{T}_{-2}$. 
\end{itemize}

\item[(A-2)]
If $S$ satisfies $N_4'=2$, 
then it satisfies $\sharp V(S)=\sharp \Delta_{\Z} -1$ 
and contains exactly two parallelograms in the following list and 
the rest of $S$ consists of triangles: 
\begin{itemize}
\item[(i)]
$\Delta^{\mathrm{par}}_4(1;1,1)$, $\Delta^{\mathrm{par}}_4(0;1,1)$ 
\item[(ii)]
two $\Delta^{\mathrm{par}}_4(0;1,1)$ with one of $\mathbb{T}_{-1}$. 
\end{itemize}

\item[(A-3)]
If $S$ satisfies $N_4'=3$, 
then $\sharp V(S)=\sharp \Delta_{\Z}$ holds and 
$S$ contains exactly three parallelograms. 
Thus $S$ has three 
$\Delta^{\mathrm{par}}_4(0;1,1)$ 
and the rest of $S$ consists of triangles whose area is $1/2$.
\end{itemize}

In the above list, by Remark~\ref{KR} and Lemma~\ref{polyt1}, 
cases 
(vi), (vii), (ix), (xi) in (A-0) 
does NOT occur. 
Furthermore, 
the following cases 
do NOT have a regular singularity by Lemma~\ref{nonreg}: 
\begin{itemize}
\item
(ii), (v), (viii), (x), (xiii), (xiv) in (A-0), 

\item 
(ii), (iii), (iv), (v) in (A-1), 

\item 
(i), (ii) in (A-2),

\item
(A-3). 
\end{itemize}
Among them, 
the refinement of the following cases 
do NOT have an irregular singularity 
by Lemma~\ref{nonirr} and Remark~\ref{onelength}: 
\begin{itemize}
\item
(ii), (v), (viii), (x), (xiii), (xiv) in (A-0), 

\item
(ii), (iii), (iv), (v) in (A-1), 

\item
(i), (ii) in (A-2),

\item
(A-3). 

\end{itemize}
The remaining cases are (i), (iii), (iv) and (xii) in (A-0) 
and (i) in (A-1), and they correspond to the polytopes
$\Delta_{\I}, \Delta_{\II}, \Delta_{\III}, \Delta_{\IV}, \Delta_{\V}$ 
and $\Delta_{\VI}$,
respectively, by Lemma~\ref{unique}.
Moreover, by Lemma~\ref{tacpiece}, \ref{tacpiece2}, \ref{singedge} 
and \ref{refined}, 
these polytopes are $1$-tacnodal.

\subsection{Case (B)}
We assume that 
$S$ is a TP-subdivision and satisfies 
$\sharp \partial \Delta_{\Z} - \sharp(V(S)\cap \partial \Delta)=1$. 
By the latter condition, 
$S$ must have exactly one polytope $P \in S$ such that 
$P \cap \partial \Delta$ is a segment of length $2$.
By Lemma~\ref{rklem}, we get  
\[ \rk(T_F)=\rkexp(T_F)=\sharp \Delta_{\Z}-4. \] 
By the definition of $\rkexp(T_F)$, we obtain 
\begin{align*}
\sharp \Delta_{\Z}-4
&=\sharp V(S)-1-\sum_{k=1}^N(\sharp V(\Delta_k)-3)\\
&=\sharp V(S)-1-N_4'.
\end{align*}
Since $\sharp V(S) \le \sharp \Delta_{\Z}-1$, we have 
$0 \le N_4' \le 2$. 
\begin{itemize}

\item[(B-0)]
If $S$ satisfies $N_4'=0$, 
then $S$ satisfies 
$\sharp V(S)=\sharp \Delta_{\Z} -3$ 
and consists of triangles. 
Let $P \in S$ be a polytope which intersects $\partial \Delta$ as 
a segment of length $2$. 
Then $S$ satisfies one of the following: 
\begin{itemize}
\item[(i)]
$P=\Delta_3(0;2,1,1)$ 
and $S$ contains two of $\mathbb{T}_{-1}$ or one of $\mathbb{T}_{-2}$, 
\item[(ii)]
$P=\Delta_3(1;2,1,1)$ 
and $S$ contains one of $\mathbb{T}_{-1}$, 
\item[(iii)]
$P=\Delta_3(2;2,1,1)$, 
\item[(iv)]
$P=\Delta_3(0;2,2,2)$, 
\item[(v)]
$P=\Delta_3(0;2,2,1)$, 
and $S$ contains one of $\mathbb{T}_{-1}$, 
\item[(vi)]
$P=\Delta_3(1;2,2,1)$, 
\item[(vii)]
$P=\Delta_3(0;2,3,1)$.  
\end{itemize}

\item[(B-1)]
If $S$ satisfies $N_4'=1$, 
then $S$ satisfies 
$\sharp V(S)=\sharp \Delta_{\Z} -2$. 
Let $P \in S$ be a polytope which intersects $\partial \Delta$ as 
a segment of length $2$. 
Then $S$ satisfies one of the following:
\begin{itemize}
\item[(i)]
$P=\Delta^{\mathrm{par}}_4(0;2,1)$ and $\Delta_3(0;2,1,1)$ 
such that their intersection $P \cap \Delta_3(0;2,1,1)$ 
is a segment of length $2$,
\item[(ii)]
$P=\Delta_3(0;2,1,1)$ 
and $S$ contains $\Delta^{\mathrm{par}}_4(1;1,1)$,
\item[(iii)]
$P=\Delta_3(1;2,1,1)$ 
and $S$ contains $\Delta^{\mathrm{par}}_4(0;1,1)$,
\item[(iv)]
$P=\Delta_3(0;2,2,1)$ 
and $S$ contains $\Delta^{\mathrm{par}}_4(0;1,1)$,
\item[(v)]
$P=\Delta_3(0;2,1,1)$ 
and $S$ contains $\Delta^{\mathrm{par}}_4(0;1,1)$,
and one of $\mathbb{T}_{-1}$. 
\end{itemize}

\item[(B-2)]
If $S$ satisfies $N_4'=2$, 
then $S$ satisfies 
$\sharp V(S)=\sharp \Delta_{\Z}$ 
and contains exactly two parallelograms. 
Thus $P=\Delta_3(0;2,1,1)$ and 
$S$ contains two $\Delta^{\mathrm{par}}_4(0;1,1)$. 
\end{itemize}

In the above list, by Lemma~\ref{polyt1}, 
the following cases do NOT occur: 
\begin{itemize}
\item
(v), (vi), (vii) in (B-0),

\item
(iv) in (B-1). 
\end{itemize} 

Furthermore, the following cases do NOT have a regular singularity 
by Remark~\ref{KR} and Lemma~\ref{nonreg}:
\begin{itemize}
\item (i), (ii), (iv) in (B-0),
\item (i), (ii), (iii), (v) in (B-1),
\item (B-2).
\end{itemize}
Among them, (iv) in (B-0) does NOT have an irregular singularity 
by Lemma~\ref{nonirr} 
and the other polytopes except (iii) in (B-0) also do NOT have it 
since they have only one edge of length more than $1$, 
which should be on the boundary $\partial \Delta$, 
and this edge cannot be a $1$-tacnodal edge.
The remaining case is (iii) in (B-0) and this corresponds 
to the polytope $\Delta_{\III}$ by Lemma~\ref{unique}.
Moreover, by Lemma~\ref{nonreg} (3), 
this polytope is NOT $1$-tacnodal.

\subsection{Case (C)}
We assume that 
$S$ is NOT a TP-subdivision. 
Then 
\begin{align*}
d(S)
&=\sharp \Delta_{\Z}-4
-\bigl\{ \sharp V(S)-1 -\sum_{k=1}^N(\sharp V(\Delta_k)-3) \bigr\}\\
&=\sharp \Delta_{\Z}-\sharp V(S)-3+\sum_{k=1}^N(\sharp V(\Delta_k)-3)\\
&\ge -3+\sum_{k=1}^N(\sharp V(\Delta_k)-3) \\
&=\sum_{m \ge 3}(m-3)N_m -3. 
\end{align*}
By $0 \le d(S) \le \mathcal{N}_S/2$ due to Lemma~\ref{rklem}, 
we get 
\[
\sum_{m \ge 3}(m-3)N_m \le -\sum_{m \ge 2} N'_{2m}+5 
\;\;\; \text{and}\;\;\;
\sum_{m \ge 2} N'_{2m} \le 2.
\]
We decompose the proof into the following three cases: 
\begin{itemize}
\item[(C-0)]
$\sum_{m \ge 2} N'_{2m}=0$ and $\sum_{m \ge 3}(m-3)N_m \le 5$, 

\item[(C-1)]
$\sum_{m \ge 2} N'_{2m}=1$ and $\sum_{m \ge 3}(m-3)N_m \le 4$, 

\item[(C-2)]
$\sum_{m \ge 2} N'_{2m}=2$ and $\sum_{m \ge 3}(m-3)N_m \le 3$. 
\end{itemize}

\begin{itemize}
\item[(C-0)]
In this case, since $N_4+2N_5+3N_6+4N_7+5N_8 \le 5$ 
and 
$\sum_{m \ge 2} N'_{2m} =0$, 
possible patterns are the following: 
\begin{itemize}
\item[(i)]
$N_8=1$ and $N'_8=0$, 

\item[(ii)]
$N_7=1$, $N_4=0,1$ and $N'_4=0$,

\item[(iii)]
$N_6=1$, $N'_4, N'_6=0$ and 
$(N_4, N_5)=(0,0), (1,0), (2,0), (0,1)$, 

\item[(iv)]
$N_5=2$, $N_4=0,1$ and $N'_4=0$, 

\item[(v)] 
$N_5=1$, $N_4=0,1,2,3$ and $N'_4=0$, 

\item[(vi)]
$N_4=1,2,3,4,5$ and $N'_4=0$.  
\end{itemize}

In case (i), $N_8=1$ and $N'_8=0$.  
Since $\mathcal{N}_S=4$, 
we get $0 \le d(S) \le 2$. 
On the other hand, any octagon has two or more inner lattice points
(Lemma~\ref{poly2}), 
so 
\begin{align*}
d(S)
&=\sharp \Delta_{\Z}-4-\{ \sharp V(S)-1-5\}\\
&=\sharp \Delta_{\Z}-\sharp V(S)+2\\
&\ge 4. 
\end{align*}
This is a contradiction. 
Therefore case (i) does not occur. 
We can prove that the above cases except the cases 
(v) with $N_4=0$ and (vi) with $N_4=1,2$ 
do NOT occur by the same argument.


\vspace{2mm}
Next, we observe the remaining cases. 

\noindent
\underline{\textit{Case (v) with $N_4=0$.}} \; 
$S$ has exactly one pentagon and 
the rest of $S$ consists of triangles. 
Then $\rkexp(S)=\sharp V(S) -3$ holds. 
Therefore, the set $(\Delta \cap \Z^2) \setminus V(S)$ 
is exactly one lattice point. 
By Lemma~\ref{poly2} (2), 
the pentagon is 
$\Delta_5(1;1,1,1,1,1)$. 
This polytope is equivalent to 
$\Delta_{\VII}$ by Lemma~\ref{unique} (6). 
Moreover, by Lemma~\ref{tacpiece2}, 
the pentagon is a $1$-tacnodal polytope. 

\vspace{2mm}
\noindent
\underline{\textit{Case (vi) with $N_4=1$.}}\;
$S$ has exactly one non-parallel quadrangle and 
the rest of $S$ consists of triangles. 
Since $\rkexp(S)=\sharp V(S)-2$, 
the set 
$(\Delta \cap \Z^2) \setminus V(S)$
consists of two lattice points. 
Therefore, a possible non-parallel quadrangle $\Delta_4(I;s,t,u,v)$ 
is one of the following list: 
\begin{itemize}
\item[(a)] $\Delta_4(0;1,1,1,1)$, 
\item[(b)] $\Delta_4(0;2,1,1,1)$, 
\item[(c)] $\Delta_4(0;2,2,1,1)$, 
\item[(d)] $\Delta_4(1;1,1,1,1)$, 
\item[(e)] $\Delta_4(1;2,1,1,1)$, 
\item[(f)] $\Delta_4(2;1,1,1,1)$. 
\end{itemize}

Cases (a) and (c) 
do NOT occur by Lemma~\ref{poly2} and Lemma~\ref{polyt1}, respectively. 
The polytopes in (b) and (e) are NOT $1$-tacnodal polytopes 
by (4) of Lemma~\ref{nonreg} and Lemma~\ref{expol_nontac}, respectively. 
Also the polytope in (d) is NOT a $1$-tacnodal polytope 
by \cite[Lemma 4.2 (\textit{i})]{S}. 
Notice that, by Remark~\ref{onelength}, 
the polytope in (d) does NOT have a $1$-tacnodal edge. 

By Lemma~\ref{unique}, 
the polytope (f) is equivalent to one of 
\[ 
\Delta_{\VIII}, \quad \Delta_{\IX} \quad \text{and} \quad  
\mathrm{Conv}\{(1,0), (0,1), (2,1), (1,3) \}.  
\]
The polytopes $\Delta_{\VIII}$, $\Delta_{\IX}$
are $1$-tacnodal polytopes by Lemma~\ref{tacpiece2}. 
On the other hand, the polytope 
$\mathrm{Conv}\{(1,0), (0,1), (2,1), (1,3) \}$
is NOT a $1$-tacnodal polytope by Lemma~\ref{nonreg} (5) 
and does NOT have a $1$-tacnodal edge by Remark~\ref{onelength}.

If $S$ contains the polytope in (b), 
since $\rkexp(S)=\sharp \Delta_{\Z} -3$, 
the adjacent polytope which shares the edge of length $2$ of 
$\Delta_4(1;2,1,1,1)$
must be either $\Delta_3(0;2,1,1)$ or $\Delta_3(1;2,1,1)$.
Each of their intersection 
with $\Delta_4(1;2,1,1,1)$ 
is NOT a $1$-tacnodal edge 
by (2) and (4) of Lemma~\ref{nonirr}. 
Therefore, any edge contained in $S$ is NOT a $1$-tacnodal edge. 

If $S$ contains the polytope (e), 
since $\rkexp(S)=\sharp \Delta_{\Z} -4=\rk(S)$, 
the adjacent polytope which shares the edge of length $2$ of 
$\Delta_4(1;2,1,1,1)$ 
must be $\Delta_3(0;2,1,1)$. 
This is a dual subdivision of a tropical 1-tacnodal curve of type (E). 

\vspace{2mm}
\noindent
\underline{\textit{Case (vi) with $N_4=2$.}} 
$S$ has exactly two non-parallel quadrangles and 
the rest of $S$ consists of triangles. 
Since $\rkexp(S)=\sharp V(S) -3$, 
the set $(\Delta \cap \Z^2) \setminus V(S)$ 
consists of exactly one lattice point. 
Therefore $S$ contains 
$\Delta_4 (0;2,1,1,1)$ and 
$\Delta'_4 (0;2,1,1,1)$ such that 
their intersection is a segment whose length is $2$. 
This is because 
a non-parallel quadrangle must satisfy either 
``the number of interior lattice points is non-zero" or 
``the polytope has an edge of length $\ge 2$", by Lemma~\ref{poly2}.  
These polytopes are NOT $1$-tacnodal polytopes by 
Lemma~\ref{nonreg}. 
Also their intersection is NOT a $1$-tacnodal edge by 
Lemma~\ref{nonirr} (5). 

\vspace{2mm}

\item[(C-1)]
In this case, since $N_4+2N_5+3N_6+4N_7 \le 4$ and 
$\sum_{m \ge 2} N'_{2m} =1$, 
the following patterns can occur:
\begin{itemize}
\item[(i)]
$N_6=N'_6=1$, $N_4=0,1$ and $N'_4=0$, 
\item[(ii)]
$N_4=N'_4=1$, $N_6=1$ and $N'_6=0$, 
\item[(iii)]
$N_4=2$, $N'_4=1$ and $N_5=1$, 
\item[(iv)]
$N_4=N'_4=1$ and $N_5=1$, 
\item[(v)]
$N_4=2,3,4$ and $N'_4=1$. 
\end{itemize}

However, we can check that the cases, except case (v) with $N_4=2$, 
are impossible by the same argument as in case (i) in (C-0).

We observe case (v) with $N_4=2$. 
$S$ contains a non-parallel quadrangle $P$ and a parallelogram $Q$, and 
the rest of $S$ consists of triangles. 
Notice that, by Lemma~\ref{poly2}, 
$P$ must satisfy either 
``the number of interior lattice points is non-zero" or 
``the polytope has an edge of length $\ge 2$". 
Since $\rkexp(S)=\sharp V(S)-3$, 
the set $(\Delta \cap \Z^2) \setminus V(S)$ 
consists of exactly one lattice point. 
Therefore $P$ and $Q$ must be either 
\begin{itemize}
\item[\textbullet]
$P=\Delta_4(1;1,1,1,1)$ and $Q=\Delta_4^{\mathrm{par}}(0;1,1)$, or 

\item[\textbullet]
$P=\Delta_4(0;2,1,1,1)$ 
and $Q=\Delta_4^{\mathrm{par}}(0;1,1)$ such that 
the edge of length $2$ of $P$ intersects the triangle 
$\Delta_3(0;2,1,1)$.  
\end{itemize}
In both cases, 
the polytopes are not $1$-tacnodal by 
Lemma~\ref{nonreg}, Lemma~\ref{nonirr} and Remark~\ref{onelength}. 

\vspace{2mm}

\item[(C-2)]
In this case, since $N_4+2N_5+3N_6 \le 3$ and 
$\sum_{m \ge 2} N'_{2m} =2$, 
any possible subdivision satisfies $N_4=3$ and $N'_4=2$. 
Since $\mathcal{N}_S=0$, 
we get $d(S)=0$. 
On the other hand, 
since $\sharp V(S) \le \sharp \Delta _{\Z}-1$ 
by Lemma \ref{poly2}, 
\[
d(S)=\sharp \Delta_{\Z}- \sharp V(S)\ge 1. 
\]
This is a contradiction. 
\end{itemize}

\subsection{Case (D)}
We assume that 
$S$ is NOT a TP-subdivision and satisfies 
$\sharp \partial \Delta_{\Z} - \sharp(V(S)\cap \partial \Delta)=1$. 
By the former condition, 
we can apply the same argument of (C) to case (D) and 
obtain the list of possible subdivisions as follows: 
\begin{itemize}
\item[(1)]
(v) with $N_4=0$ in (C-0),

\item[(2)]
(vi) with $N_4=1$ in (C-0), 

\item[(3)]
(vi) with $N_4=2$ in (C-0), 

\item[(4)]
(v) with $N_4=2$ in (C-1). 

\end{itemize}

\noindent
\underline{\textit{Case (1).}} 
$S$ has exactly one pentagon and 
the rest of $S$ consists of triangles. 
Then $\rkexp(S)=\sharp V(S) -3$ holds. 
By the boundary condition,  
the set $(\Delta \cap \Z^2) \setminus V(S)$ is empty. 
If $S$ contains a triangle $P$ whose intersection with $\partial \Delta$ 
is an edge of length $2$, 
then, by Lemma~\ref{poly2}, 
$S$ does NOT have a pentagon. 
Therefore, the possible pentagon is 
$\Delta_5(0;2,1,1,1,1)$, 
whose intersection with $\partial \Delta$ is an edge of length $2$. 
However, the pentagon does NOT exist by Lemma~\ref{polyt1}. 

\vspace{2mm}
\noindent
\underline{\textit{Case (2).}}
$S$ has exactly one non-parallel quadrangle and 
the rest of $S$ consists of triangles. 
By $\rkexp(S)=\sharp V(S)-2$ and the boundary condition, 
the set 
$(\Delta \cap \Z^2) \setminus V(S)$
consists of one lattice point. 
Therefore, possible non-parallel quadrangle $\Delta_4(I;s,t,u,v)$ 
is one of the following list: 
\begin{itemize}
\item[(a)] $\Delta_4(0;1,1,1,1)$, 
\item[(b)] $\Delta_4(0;2,1,1,1)$, 
\item[(c)] $\Delta_4(1;1,1,1,1)$, 
\end{itemize}

Case (a) does NOT occur 
by Lemma~\ref{poly2}. 
The polytope in (b) is NOT a $1$-tacnodal polytope 
by Lemma~\ref{nonreg} (4). 
Also the polytope in (c) is NOT a $1$-tacnodal polytope 
by \cite[Lemma 4.2 (\textit{i})]{S}. 
Notice that, by Remark~\ref{onelength}, 
the polytope in (c) does NOT have a $1$-tacnodal edge. 

If $S$ contains the polytope in (b), 
since $\rkexp(S)=\sharp \Delta_{\Z} -4$, 
the intersection of the quadrangle $\Delta_4(0;2,1,1,1)$ and 
$\partial \Delta$ is an edge of length $2$.  
Thus the edge is NOT a $1$-tacnodal edge. 

\vspace{2mm}
\noindent
\underline{\textit{Case (3) and (4).}}
$S$ has exactly two non-parallel quadrangles and 
the rest of $S$ consists of triangles. 
By $\rkexp(S)=\sharp V(S) -3$ and 
the boundary condition, 
the set $(\Delta \cap \Z^2) \setminus V(S)$ is empty. 
Therefore, such subdivision $S$ does NOT exist
by the fact that a non-parallel quadrangle must satisfy either 
``the number of interior lattice points is non-zero" or 
``the polytope has an edge of length $\ge 2$" in Lemma~\ref{poly2}.  
Case (4) can be proved by the same argument.  \qed
\vspace{3mm}

\begin{rem}\label{further}
As mentioned in the introduction, 
this research aims to construct 
the tropical version of enumerative geometry of the 1-tacnodal curves.
Therefore, we would like to lift the 1-tacnodal curve 
from a given degenerate 1-tacnodal curve by patchworking.
It is known that there is no obstruction if the singular point is $A_1$, 
and this is still true even if it is $A_2$, 
which can be checked by a numerical criterion of 
the vanishing of the obstruction constructed by Shustin 
(See \cite[Theorem 4.1]{Sglu}, or \cite[Lemma 5.4]{S} 
for a tropical version).
But, unfortunately, this criterion does not work if it is $A_3$ 
because of the following reason:

We recall a sufficient condition to apply patchworking 
\cite[Lemma 5.5 (\textit{ii})]{S}, called \textit{transversality}. 
Let $S$ be the dual subdivision of a tropical curve $T$, 
$\Delta_1, \ldots, \Delta_N$ 
be the $2$-dimensional polytopes of $S$ and 
$(C_1, \ldots, C_N)$ 
be a collection of complex curves such that 
the Newton polytope of the defining polynomial $f_i$ of $C_i$ 
is $\Delta_i \in S$ and, if 
$\sigma_{ij}:=\Delta_i \cap \Delta_j \neq \emptyset$, 
$f_i^{\sigma_{ij}}=f_j^{\sigma_{ij}}$.

For an irreducible curve $C_k$ for some $k \in \{1,\ldots N\}$, 
there is a union $\Delta_k^-$ of edges of $\Delta_k$ 
such that $C_k$ satisfies the following inequality: 
\[
{\sum}'b(C_k,\xi)
+{\sum}'' \tilde{b}(C_k, Q) 
+{\sum}'''\bigl( (C_k \cdot X(\sigma))-\epsilon \bigr)
<
\sum_{\sigma \subset \partial \Delta}(C_k \cdot X(\sigma)),
\]
where 
\begin{itemize}
\item 
if $C$ has a tacnode, then $b(C,\xi) = 1$ for both branches, 
if $C$ is locally given by 
$\{ x^{pr} + y^{qr} = 0 \}$ for coprime integers $p,q$, 
then 
$\tilde{b}(C,\xi) = p+q-1$ for each branch,

\item
${\sum}'$ ranges over all local branches $\xi$ of $C_k$, 
centered at the points $z \in \mathrm{Sing}(C_k)\cap (\C^*)^2$,

\item
${\sum}''$ ranges over all local branches $Q$ of $C_k$, 
centered at the points
$z \in \mathrm{Sing}(C_k) \cap X(\partial \Delta_k)$, and 

\item
${\sum}'''$ ranges over all non-singular points $z$ of $C_k$ on 
$X(\partial \Delta_k)$ with $\epsilon = 0$ if 
$\sigma \subset \Delta_k^-$ 
and $\epsilon=1$ otherwise,
\end{itemize}
then $C_k$ is transversal.

Let $V \subset X(\Delta_{\III})$ be a curve 
which is constructed in Lemma~\ref{singedge}. 
We can easily check 
\[
{\sum}'b(V,\xi) = 0,
\quad
{\sum}'' \tilde{b}(V, Q) = 4, 
\quad
{\sum}'''\bigl( (V \cdot X(\sigma))-\epsilon \bigr) \ge 0 
\quad
\text{and}
\sum_{\sigma \subset \partial \Delta}(V \cdot X(\sigma))
=4. 
\]
Therefore $V$ does not satisfy the above inequality. 
\end{rem}


\end{document}